\documentclass[%
 amsmath,amssymb,
]{article}
\usepackage[utf8]{inputenc}
\usepackage[T1]{fontenc}
\usepackage{authblk}

\usepackage{fullpage}
\usepackage{authblk}
\usepackage[a4paper,]{geometry}
\usepackage{amsmath}
\usepackage{graphicx}
\usepackage[]{latexsym,amssymb,color,subfigure,amsthm,}
\usepackage{diagbox}
\usepackage{verbatim}
\usepackage{dcolumn}
\usepackage{bm}
\allowdisplaybreaks
\usepackage{enumerate}
\usepackage{float}
\usepackage{tikz}
\newtheorem{theorem}{Theorem}[section]
\newtheorem{corollary}{Corollary}

\newtheorem{lemma}[theorem]{Lemma}
\newtheorem{proposition}{Proposition}
\newtheorem{conjecture}{Conjecture}

\theoremstyle{definition}
\newtheorem{definition}[theorem]{Definition}
\newtheorem{remark}{Remark}

\newtheorem{example}{Example}[section]

\newcommand{\subG}{{_{G}\,}}

\begin{document}
%\preprint{AIP/123-QED}
%\preprint{}

\title{Geometric Model of the Fracture as a Manifold Immersed in Porous Media} %Title of paper



%XXXXXXXXXXXXXXXXXXXXXXXXXXXXXXXXXXXXXX

\author[1]{Pushpi Paranamana \thanks{pushpi.paranamana@rutgers.edu} }
\author[2]{Eugenio Aulisa \thanks{eugenio.aulisa@ttu.edu}}
\author[2]{Akif Ibragimov \thanks{akif.ibraguimov@ttu.edu}}
\author[2]{Magdalena Toda \thanks{magda.toda@ttu.edu}}
\affil[1]{\small{Department of Mathematics and Computer Science, Rutgers University-Newark, NJ, 07102, USA}}
\affil[2]{\small{Department of Mathematics and Statistics, Texas Tech University, Lubbock, TX, 79409-1042, USA}}

%\author{Pushpi Paranamana \thanks{pushpi.paranamana@rutgers.edu} }\affil{Department of Mathematics and Computer Science\\Rutgers University-Newark\\ NJ, 07102, USA}
%%\email[]{pushpi.paranamana@rutgers.edu}
%%%\homepage[]{Your web page}
%%%\thanks{}
%%%\altaffiliation{}
%\author{Eugenio Aulisa \thanks{eugenio.aulisa@ttu.edu}}
%%\homepage[]{Your web page}
%%\thanks{}
%%\altaffiliation{}
%\affiliation{Department of Mathematics and Statistics, Texas Tech University, Lubbock, TX, 79409-1042, USA}
%
%\author{Akif Ibragimov}
%\email[]{akif.ibraguimov@ttu.edu}
%%\homepage[]{Your web page}
%%\thanks{}
%%\altaffiliation{}
%\affiliation{Department of Mathematics and Statistics, Texas Tech University, Lubbock, TX, 79409-1042, USA}
%
%\author{Magdalena Toda}
%\email[]{magda.toda@ttu.edu}
%%\homepage[]{Your web page}
%%\thanks{}
%%\altaffiliation{}
%\affiliation{Department of Mathematics and Statistics, Texas Tech University, Lubbock, TX, 79409-1042, USA}
%XXXXXXXXXXXXXXXXXXXXXXXXXXXXXXXXXXXXXX
%\subjclass{Primary: 35J25, 76S05; Secondary: 76D.}
%\keywords{non-linear Darcy-Forchheimer equations, fracture modeling, asymptotic convergence, optimization, diffusive capacity.}

\vspace{-8ex}
  \date{}
\maketitle 
\begin{abstract}
In this work, we analyze the flow filtration process of slightly compressible fluids in porous media containing  man made fractures with complex geometries. We model the coupled fracture-porous media system where the linear Darcy flow is considered in porous media and the nonlinear Forchheimer equation is used inside the fracture. 

We develop a model to examine the flow inside fractures with complex geometries and variable thickness, on a Riemannian manifold. The fracture is represented as the normal variation of a surface immersed in $\mathbb{R}^3$. Using operators of Laplace Beltrami type and geometric identities, we model an equation that describes the flow in the fracture. A reduced model is obtained as a low dimensional BVP. We then couple the model with the porous media. 

Theoretical and numerical analysis have been performed to compare the solutions between the original geometric model and the reduced model in reservoirs containing fractures with complex geometries. We prove that the two solutions are close, and therefore, the reduced model can be effectively used in large scale simulators for long and thin fractures with complicated geometry.

\end{abstract}

\section{Introduction}
\begin{comment}
(At present the combination of the multistage hydraulic fracturing and horizontal wells has become a widely used technology in simulating tight oil reservoirs. However the ideal planar fractures used in the reservoir simulation are simplified excessively. Effects of some key fracture properties (e.g., fracture-geometry distributions and the permeability variations) are not usually taken into consideration during the simulation. Oversimplified fractures in the reservoir model may fail to represent the complex fractures in reality, leading to significant errors in forecasting the reservoir performance. In this paper, we examined the different fracture-geometry distributions and discussed the effects of geometry distribution on well production further. 
--Effect of Fracture Geometry on Well Production in Hydraulic-Fractured Tight Oil Reservoirs. 
)
\end{comment}

Fractured reservoir modeling is a multi-scale and multi-physics complex problem, where analysis and simulation of this complex system require a deep understanding of the physical processes that describe the flow coupling at different scales. We consider flows in the porous media which include fractures of the same dimension as the reservoir itself. In practice fracture domain is very ``long'' compared to its width, and require a special care from analytical  and numerical point of view. Moreover, the geometry of the fracture domain is a significant characteristic. Effects of complex geometric features of fractures such as variable thickness (which is very small) and the curvature of the fracture's boundary  are usually not taken into consideration during the simulations, which prevents successful reservoir modeling and leads to errors in forecasting the reservoir performance. 

The traditional approach in the reservoir modeling is to simplify the problem by treating the fracture as a 1-D sink with pressure on the fracture equals to the value of the pressure on the well or
with given flux, and the flow occurs in the porous media surrounding the fracture.
Mathematically, this is formulated as a Dirichlet or Neumann boundary condition on the fracture for pressure function, which satisfies an elliptic or parabolic equation of second order. This assumption leads to a significant overestimation of the fracture
capacity. 

In our previous paper \cite{paranamana2019fracture}, we consider the nonlinear nature of the flow inside the fracture and couple it with the flow in the porous media. The main goal of this paper is to incorporate complex geometric features of the fractures into the modeling and explore the pressure distribution of the flow. We develop a framework based on methods of differential geometry to model the fractures with complicated geometries. In our approach, we formulate the
fracture as a 3-D manifold immersed in porous media and we introduce a reduced dimensional nonlinear flow equation inside the fracture. 

%\RED{Henceforth for convenience we will call the domain of the fracture as \textit{fracture}}.

We consider slightly compressible fluid flows in the inhomogeneous reservoir fracture domain.  
Due to the heterogeneity of the fractured porous media, the velocity of the flow inside the fracture is much higher than inside porous media. It was observed that, due to high velocity inside fractures, the inertial effects become significant.
Therefore the relation between the flow velocity and the pressure gradient inside the fractures deviates from Darcy's law \cite{bear2013dynamics, forchheimer1901wasserbewegung, dake1983fundamentals}. 

We consider Darcy's law to model the flow in porous media and a generalized nonlinear Forchheimer equation to describe the flow inside the fracture. 
%%%%%%%%%%%%%%%%%%%%%%%%%%%%%%%%%
It is evident that the fluid mostly flows towards the fracture first and then transports to the well along the fracture. Consequently, the total production of the hydrocarbon depends on the
capacity of the fracture to \textit{take-in} fluid from the reservoir. This capacity depends on the geometry of the fracture and its conductivity, and it is characterized by the productivity
index of the reservoir fracture system.

In past we proved that an integral functional called diffusive capacity, defined as the total flux on the well
surface divided by the pressure drawdown (difference between average pressures of the
reservoir domain and the well boundary) is mathematically equal
to productivity index \cite{aulisa2009mathematical}. We use the diffusive capacity to characterize the reservoir performance.

We formulate a model to investigate the flow inside fractures with complex geometries. In particular, the fracture is represented as a perturbation across the fracture thickness in the normal direction to the barycentric surface immersed in $\mathbb{R}^3$, with its naturally induced Riemannian metric. Moreover, we formulate the flow equation inside the fracture using the first and second fundamental forms of the surface, geometric identities for Gaussian curvature and mean curvature, and corresponding operators of the Laplace-Beltrami type. 
On the boundary of the domain of the flow (union of porous media and fracture) we impose mixed boundary conditions. 
%The flow is following towards the well, is boundary of the fracture, and other boundary of the domain considered to be isolated. 
On the well we impose Dirichlet condition. 
%Arising 3-D boundary value problem  in the union of the porous domain and the fracture we call the \textit{actual} problem. 
This coupled fluid flow is impossible to solve numerically, since the thickness of the fracture is of the $10^6$ smaller than the length of the fracture. Therefore  we introduce a reduced model for the flow inside the fracture domain.
% which model 3-D porous media and fracture will become part of the boundary of 2-D dimension. Boundary condition is formulated in terms of specific differential equation which   inherits 3-D flow equation the fracture of thickness $h$, with parameter $h$ explicitly presented in the R.H.S. of the equation.  We will call corresponding problem as  a \textit{reduced}  BVP.} 
Moreover, we obtain further simplified models with further assumptions (that can be utilized depending on the size of the fracture thickness and other physical factors).

\begin{comment}Using the reduced model, pressure distributions of nonlinear flows for different fracture geometries were obtained numerically, and compared with the  pressure distributions of actual problem. Reduced  model provides very similar solutions to actual solutions, allowing the validation of the fracture model.  We observe that the
reduced model works best for thin fractures.
Furthermore, using the reduced  models we evaluate the pressure distributions and the diffusive capacities for coupled domain for different system of fracture geometries in various reservoir geometries. We compare those solutions with the solutions obtained  from actual high dimensional flow equations. The two solutions are almost the same for realistic values of thickness, thus confirming the successful implementation of the model.\end{comment}

We theoretically and numerically investigate the difference
between the solutions of the actual model and the reduced models. We confirm the successful implementation of the models, by proving that the solutions of the reduced models are close to the solutions of the actual model, for realistic values of fracture thickness.

Controlling the shape of the fractures in geological reservoirs is challenging. Therefore, our method can be applied mostly for simple fracture geometries. However, the geometric method and the analysis we introduce in this paper are valuable tools in modeling micro fluidic flows and blood flows in arteries and veins \cite{aulisa2017fluid, somaweera2015chip}. Moreover, we believe that this methodology can be served as a foundation for reservoir engineers to model fractures in the future.
% with new innovative technology.

%\RED{We want to point out the following. Controlling the shape of fractures in geological reservoirs is a challenging problem. Therefore, our method can be applied mostly for simple fracture geometries. However, the geometric method and the analysis we introduce in this paper are valuable tools in modeling micro fluidic flows and blood flow in arteries and veins %\cite{microfluidic, blood}. Moreover, we believe that this methodology can be served as a foundation for reservoir engineers to model fractures in the future.
% with new innovative technology.}

\section{Formulation of the problem and preliminary results}

In this section, we summarize important preliminary results on Darcy-Forchheimer equations from our paper \cite{paranamana2019fracture}.

\subsection{Reservoir modeling}
Mathematical framework of the reservoir modeling is based on Darcy-Forchheimer equation, the continuity equation and the state equation \cite{aulisa2009mathematical, dake1983fundamentals, muskat1938flow}. Among various methods that demonstrate non-Darcy case, 
%the general Brinkman-Forchheimer
non-linear Forchheimer equation 
%is considered to be the most appropriate 
is widely utilized \cite{ewing1999numerical, douglas1993generalized, forchheimer1901wasserbewegung, payne1999convergence}. 

The velocity vector field $\mathbf{v}$ and the pressure $p$ in porous media are related by the 
%time dependent Brinkman-Forchheimer 
Forchheimer equation given by
%\begin{equation}
%\left(\rho c_a \frac{\partial \mathbf{v}}{\partial t}- \mu \Delta \mathbf{v} \right)+\frac{\mu}{k}\mathbf{v}+\beta |\mathbf{v}|\mathbf{v}=-\nabla p ,\mbox{ with} \,
%\beta=\frac{F\rho \phi}{k^{1/2}} \label{B-F eqn}\,\,, \end{equation}
%where $c_a$ is the acceleration coefficient, $F$ is the Forchheimer coefficient, $\phi$ is the porosity, $k$ is the permeability, $\mu$ is the viscosity and $\rho$ is the density of the fluid.
%This describes the momentum conservation of the flow. 
\begin{equation}
\frac{\mu}{k}\mathbf{v}+\beta |\mathbf{v}|\mathbf{v}=-\nabla p ,\,\,\end{equation}
where $k$ is the permeability, $\mu$ is the viscosity and $\rho$ is the density of the fluid.
This describes the momentum conservation of the flow. 
\begin{remark}
In our intended application the parameters of the reservoir and the
fracture are isotropic and space dependent. Namely, $k = k_p$ and $\beta= 0$ in the
porous block, and $k = k_f$ and $\beta\neq 0$ in the fracture. 
\end{remark}

The continuity equation takes the form
\begin{equation}
\frac{\partial \rho}{\partial t}+\nabla \cdot (\rho \mathbf{v})=0\,\,.
\end{equation}

For slightly compressible fluids, the state equation is given by
 \begin{equation} \label{slightly com}
	\rho^{\prime}=\gamma^{-1}\rho \,\,\,\,\,\,\left(\rho=\rho_0 \exp^{{\gamma^{-1}}(p-p_0)}\right)\,\,,
		\end{equation}
		where $\gamma$ is the compressibility constant of the fluid.

%\begin{assumption}\label{dissipation}
In natural reservoirs, the dissipation in the porous media is dominant \cite{muskat1938flow}. We assume that the permeability coefficient $k$ is very small. Moreover, for many slightly compressible fluids $\gamma$ is of order $10^8$. \cite{aronson1986porous, dake1983fundamentals, bear2013dynamics}%Therefore we assume that the term $\rho^{\prime}\mathbf{v}\cdot\nabla p$ in Eq.\eqref{eqn} is negligible.
%\end{assumption}

%From assumption \ref{dissipation} and Eq.~\eqref{slightly com}, we obtain
%\begin{align}
%\rho^{\prime}\frac{\partial p}{\partial t}=-\rho \nabla\cdot\mathbf{v}-\rho^{\prime}\mathbf{v}\cdot\nabla p\,\,,\label{eqn}\\
%-\nabla p-\frac{\mu}{k}\mathbf{v}-\beta|\mathbf{v}|\mathbf{v}=0\,\,,\\
%\rho^{\prime}=\gamma^{-1}\rho \label{state eqn}\,\,.
%\end{align}

%\begin{assumption}\label{gamma small}
%Since for many slightly compressible liquids $\gamma^{-1}$ is very small \cite{aronson1986porous, dake1983fundamentals}, the term $\rho^{\prime}\mathbf{v}\cdot\nabla p$ in Eq.\eqref{eqn} is negligible.
%\end{assumption}
With these constraints, the system can be rewritten as \cite{paranamana2019fracture}
\begin{align}
\frac{\partial p}{\partial t}=-\gamma \nabla\cdot\mathbf{v}\,\,,\label{continuity eqn}\\
-\nabla p-\frac{\mu}{k}\mathbf{v}-\beta|\mathbf{v}|\mathbf{v}=0 \label{F eqn}\,\,.
\end{align}
\textbf{Darcy-Forchheimer equation:}
The velocity vector field $\mathbf{v}$ can be uniquely represented as a function of the pressure gradient as follows.
\begin{align} 
\mathbf{v} = \mathbf{v}_\beta =  -f_\beta \left
(\|\nabla p\|\right) \nabla p \, , \label{DF eqn} \\
f_\beta\left(\|\nabla p\|\right) = \frac{2}{\alpha+\sqrt {\alpha^2 +4 \beta \|\nabla p\|}} \,. \label{non lin term}
\end{align}
where
$\alpha=\frac{\mu}{k}.$\\
Eq.\eqref{DF eqn} is referred as Darcy-Forchheimer equation.
%\begin{lemma} \label{fbeta}
The velocity defined in Darcy Forchheimer equation \eqref{DF eqn} with $f_\beta$ defined by Eq.\eqref{non lin term}, solves the Forchheimer equation \eqref{F eqn} \cite{paranamana2019fracture}.
%\end{lemma}
%\begin{proof} 
%Proof is followed from Lemma 2.1 in Ref.\cite{aulisa2009mathematical}.
%\end{proof}
\begin{comment}
\begin{proof}
	Using Eq.\eqref{DF eqn}, Eq\eqref{F eqn} can be re written as 
	\begin{equation}
	\left(\beta \left(f_\beta(\|\nabla p\|)\right)^2\|\nabla p\|+\alpha\left(f_\beta(\|\nabla p\|)\right)-1\right) \nabla p=0.
	\end{equation}
	Let $\zeta =\|\nabla p\|\neq0 $. Then\\
	\begin{equation}
	\beta\zeta \left(f_\beta(\zeta)\right)^2+\alpha\left(f_\beta(\zeta)\right)-1=0.
	\end{equation}
	The positive root is 
	\begin{equation}
	f_\beta(\zeta) = \frac{-\alpha+\sqrt{\alpha^2+4\beta\zeta}}{2\beta\zeta}.
	\end{equation}
	Multiplying both denominator and numerator by $\alpha+\sqrt{\alpha^2+4\beta\zeta}$
	\begin{equation}
	f_\beta(\zeta)= \frac{2}{\alpha+\sqrt {\alpha^2 +4 \beta \|\nabla p\|}}.
	\end{equation}
\end{proof}
\end{comment}

%\begin{assumption}\label{beta indep}
	The coefficient $\beta$ in the non linear term of the Eq.\eqref{DF eqn} does not depend on pressure. Changes of density of slightly compressible fluids have a minor impact on changes in the coefficient $\beta$ \cite{dake1983fundamentals}.
%\end{assumption}

\begin{remark}
As $\beta \to 0$ the Darcy-Forchheimer equation reduces to Darcy equation. 
\end{remark}
Function $f_{\beta}$ has important monotonic properties.
\begin{lemma}\label{NSM} 
For $f_{\beta}(\|\eta \|)$ defined as above, the function $f_{\beta}(\|\eta \|)\|\eta \|$ is strictly monotonic on bounded sets. More precisely,
\begin{align}
\Big(f_{\beta}(\|\eta_1 \|) \eta_1 -& f_{\beta}( \|\eta_2 \|) \eta_2 \Big) \cdot (\eta_1 - \eta_2 )  \nonumber \\
& \geq \frac{1}{2}f_{\beta}\big(\max (\|{\eta}_1\| ,\|{\eta}_2\| )\big)\|{\eta}_1 - {\eta}_2 \| ^2\,.
\end{align}
\end{lemma}
%\begin{proof}
The proof follows from Lemma 2.4 in Ref. \cite{hoang2012qualitative} with $a=1/2$.
%\end{proof}

\begin{lemma}
Let $f_\beta$ be defined by the formula \eqref{DF eqn}. With the above assumptions, the pressure function $p$ satisfies the quasi linear parabolic equation
	\begin{equation}
	\frac{\partial p}{\partial t}=\gamma \nabla \cdot \left(f_\beta(\|\nabla p\|)\nabla p\right) \label{parabolic eqn}.
	\end{equation}
\end{lemma}

%\begin{proof}
Substituting Eq.~\eqref{DF eqn} to the Continuity equation~\eqref{continuity eqn}, Eq.~\eqref{parabolic eqn} can be obtained.
%\end{proof}
For more details see \cite{paranamana2018analytical}
%%%%%%%%%%%%%%%%%%%%%%%%%%%%%%%%%%%%%%%%%%%
\subsection{Diffusive Capacity and Pseudo-Steady State regime (PSS)}
System \eqref{continuity eqn} - \eqref{F eqn} characterizes the fluid flow of well exploitation in a reservoir. %Various production regimes can be determined by different boundary conditions. In our modeling, no flux is coming into or out of the exterior boundary of the reservoir. The IBVP for the system~\eqref{continuity eqn} - \eqref{DF eqn} is formulated by the well boundary condition, impermeable exterior boundary condition and the initial well pressure. The solution of the IBVP illustrates the hydrodynamical properties of the reservoir well system.
Analogous to the reservoir engineering concept of productivity index (PI) that is used to measure the capacity and the performance of the well, we introduce the mathematical notion called the diffusive capacity.

Let $\Omega$ be a bounded reservoir domain bounded by the exterior no flux boundary $\Gamma_{out}$ and the well surface $\Gamma_w$. Let $\mathbf{n}$ be the outward unit normal on the piece wise smooth surface $\Gamma_w$.

\begin{definition}Diffusive Capacity.\\
Let the pressure $p$ and the velocity $\mathbf{v}$ form the solution for the system~\eqref{continuity eqn} - \eqref{F eqn} in $\Omega$, with impermeable boundary condition $\mathbf{v}\cdot \mathbf{n} \, \big \rvert_{\Gamma_{out}} = 0$. Then the diffusive capacity is defined by
\begin{equation}
J_p(t)=\frac{\int_{\Gamma_w}\mathbf{v}\cdot \mathbf{n}\, ds }{\bar p_{\Omega}(t)- \bar p_w(t)}\,,\label{DC} 
\end{equation}
where, 
$\bar p_{\Omega}(t)-\bar p_w(t)$ is called the pressure draw down (PDD) on the well, $\bar p_{\Omega}(t)=\frac{1}{\left|\Omega\right|}\int_{\Omega}p\,d\Omega,\,\bar p_w(t)=\frac{1}{\left|\Gamma_w\right|} \int_{\Gamma_w} p\, ds,$\,$\left|\Omega\right|$ is the volume of the reservoir and
$\left|\Gamma_w\right|$is the area of the well.
\end{definition}
It has been observed on field data that when the well production rate $Q(t)= Q= const$, the well productivity index stabilizes to a constant value over time \cite{raghavan1993well}.
\begin{definition} PSS regime.\\
	Let the well production rate Q be time independent: $\int_{\Gamma_w} \mathbf{v}\cdot \mathbf{n}\, ds= Q.$ The flow regime is called a pseudo - steady state (PSS) regime, if the pressure draw down (PDD) = $\bar p_{\Omega}(t)-\bar p_w(t)$ is constant.
	\end{definition}
	
\begin{corollary}
For a PSS regime the diffusive capacity/PI is time invariant.
\end{corollary}
	
\subsubsection{PSS solution for the Initial Boundary Value Problem.}
Let $\Omega$ be a bounded reservoir domain with impermeable exterior boundary and $p(x,t_0)=p(x_0)$ be the given initial pressure in the reservoir. Assume that the well is operating under time independent constant rate of production Q and the initial reservoir pressure is known.% Let the assumptions \ref{dissipation} - \ref{gamma small} hold. 
Then the IBVP that models the oil filtration process can be formulated as 
%\begin{align}
%\rho{\prime}\frac{\partial p}{\partial t} =-\rho \nabla\cdot\mathbf{v} \,,\\
%-\nabla p-\frac{\mu}{k}\mathbf{v}-\beta|\mathbf{v}|\mathbf{v}=0\,,\\
%\int_{\Gamma_w} \mathbf{v}\cdot \mathbf{n} ds= Q\,,\\
%\mathbf{v}\cdot \mathbf{n} \, \big \rvert_{\Gamma_{out}} = 0\,,\\
%p(x,t_0)= p_0(x)\,.
%\end{align}
%Using the assumption \ref{beta indep} , the IBVP reduces to \\
\begin{align}
\gamma^{-1}\frac{\partial p}{\partial t}= \nabla \cdot \left(f_\beta(\|\nabla p\|)\nabla p\right)\label{parab eqn}\,,\\
\int_{\Gamma_w} f_\beta \left(\|\nabla p\|\right)\frac{\partial p}{\partial \mathbf{n}}  ds=-Q \label{total flux}\,,\\
\frac{\partial p}{\partial \mathbf{n}}\Big \rvert_{\Gamma_{out}} = 0\,,\\
p(x,t_0)= p_0(x) \label{ini pressure}\,.
\end{align}

Since the boundary condition on $\Gamma_w$ is a single integral condition for the total flux, the IBVP \eqref{parab eqn}-\eqref{ini pressure} has infinitely many solutions. So, we constrain the solution to an auxiliary problem by assuming the solution to be constant in space on $\Gamma_w$. A more general situation has been considered in \cite{aulisa2012time}.

\subsubsection{Steady state auxiliary BVP}
Let $Q$ be the rate of production. Assume that the boundary $\partial \Omega$ is smooth. Let $W$ be the solution of the auxiliary steady state BVP:
\begin{align}
-\nabla \cdot \left(f_\beta(\|\nabla W\|)\nabla W \right)&=\frac{Q}{\left|\Omega\right|}\quad \text{in}\, \Omega \,,\label{aux prob}\\
W \big \rvert_{\Gamma_w}=0\,,\\
\frac{\partial W}{\partial \mathbf{n}}\Big \rvert_{\Gamma_{out}} = 0 \label{aux ex boundary}\,.
\end{align}
Through integration by parts we can obtain
\begin{equation}
\int_{\Gamma_w} f_\beta(\|\nabla W\|)\frac{\partial W}{\partial \mathbf{n}}  ds=-Q\,.
\end{equation}

\begin{proposition} \label{soln to aux prob}
	Let $W(x)$ be the solution of the auxiliary problem \eqref{aux prob} - \eqref{aux ex boundary}, then
	\begin{equation}
	p(x,t)=W(x)-\gamma A t + K \label{soln}\,,
	\end{equation}
	where $ A= \frac{Q}{\left|\Omega\right|}$, solves the IBVP \eqref{parab eqn} - \eqref{ini pressure}. For this solution we have the Pseudo Steady State and therefore the diffusive capacity is constant.
\end{proposition}
\begin{proof}
Proposition \ref{soln to aux prob} is followed by substituting \eqref{soln} to Eq.\eqref{parab eqn} and verifying boundary and initial conditions \eqref{total flux} - \eqref{ini pressure}.
\end{proof}
The constant $K$ in Eq. \eqref{soln} is a measure of the initial oil reserve.
With Definition \eqref{non lin term} and monotonicity property in Lemma \ref{NSM}, the BVP \eqref{aux prob}-\eqref{aux ex boundary} has a unique weak solution $W(x)$ belonging to $\mathbb{W}^{1,\frac{3}{2}}(\Omega)$ \cite{aulisa2011long}. However, hereafter we only consider solutions $W(x)$ that satisfy $W\in C^2(\bar{\Omega})$ and $W, \nabla W \in C^2(\bar{\Omega})$ \cite{hoang2012qualitative}.

\subsection{Fractured reservoir modeling}
In this section, we introduce a fracture to the reservoir domain. We model the oil filtration process with a constant rate of production $Q$, for a fractured reservoir system. (For more details please see \cite{paranamana2019fracture}.)
Consider a fractured reservoir domain where the exterior boundary  $\Gamma_{out}$ is impermeable.
Let $\Omega_p$ be the porous media domain, $\Omega_f$ be the fracture domain, $\Gamma_f$ be the boundary between the fracture and the porous media, and $\Gamma_{f_{out}}$ be the extreme of the fracture.

\begin{figure}[h]
	\begin{center}
		\includegraphics[scale=0.35]{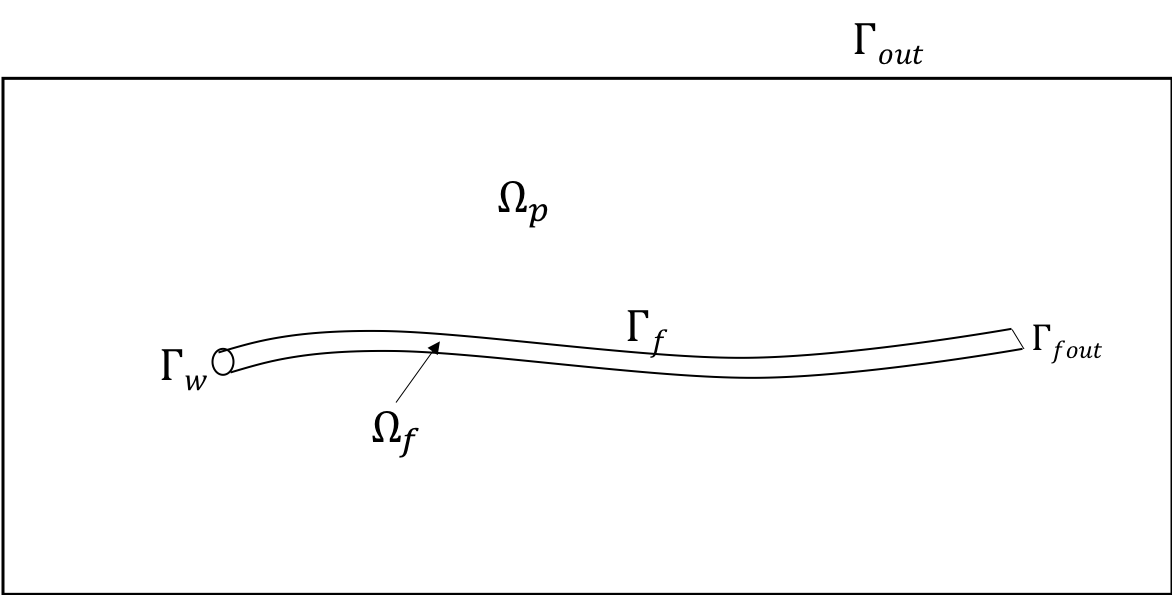}
		\caption{Fractured-Reservoir Domain}	
	\end{center}
\end{figure}
\vspace{-5mm}
Let $\mathbf{v}_p, W_p, k_p$ be the velocity, the pressure and the permeability of the flow in porous media, respectively. Let $\mathbf{v}_f, W_f$ be the velocity and the pressure of flow inside the fracture, respectively.
Let $\mathbf{n}_p$, and $\mathbf{n}_f$ be the unit outward normal vectors 
to the porous medium and the fracture, respectively.

The auxiliary BVP \eqref{aux prob} - \eqref{aux ex boundary} for the above mentioned reservoir-fracture 
system can be modeled as:
\begin{align}
- \nabla \cdot k_p \nabla W_p & =\frac{Q}{\left|\Omega\right|} &\mbox{ in } \Omega_p
 \label{porous pressure}\,,\\
- \nabla \cdot f_\beta(\|\nabla W_f\|) \nabla W_f&=\frac{Q}{\left|\Omega\right|} & \mbox{ in } \Omega_f \label{frac pressure}\,,\\
\mathbf{v}_p \cdot \mathbf{n}_p&=0  & \mbox{ on } \Gamma_{out} \, ,\\
W_p &= W_f  & \mbox{ on } \Gamma_f \cup \Gamma_{f_{out}} \label{mass_cont}\,,\\
\mathbf{v}_p \cdot \mathbf{n}_p &= - \mathbf{v}_f \cdot \mathbf{n}_f  & \mbox{ on } \Gamma_f \cup \Gamma_{f_{out}}\, ,\label{flux_cont}\\
W&=0  & \mbox{ on } \Gamma_w \label{ini well} \,.
\end{align}
Eqs.~\eqref{mass_cont} and \eqref{flux_cont} assure the continuity of the solutions and the continuity of the fluxes
across the interface $\Gamma_f$, respectively. Notice that we consider the linear Darcy law inside the reservoir while the non-linear Forchheimer equation is considered inside the fracture.
% 
% \begin{assumption}\label{cty on bdry}
% 		Continuity of the solution on the boundary of the fracture.\\
% Assume that $ W_p = W_f $ \mbox{ on } $\Gamma_f $ and flux flowing in to the fracture is equal to the flux flowing out of the fracture boundary. So $\mathbf{v}_p \cdot \mathbf{n}_p = - \mathbf{v}_f \cdot \mathbf{n}_f  \mbox{ on } \Gamma_f$
% \end{assumption}
%%%%%%%%%%%%%%%%%%%%%%%%%%%%%%%%%%%%%%%%%%%%%%%%%%%%%%%%%%%%%%
\section{Geometric modeling of the fracture} 
Fractures in porous media have very complicated geometry. The domain of the fracture is very long compared to its thickness. The fracture thickness is changing along the length and the fracture boundary has curvature. As discussed in the introduction, neglecting those geometric features lead to over estimation of the fracture capacity. In this section, we employ methods in differential geometry to model the fractures with complex geometries. Then we obtain an equation for flow inside these fractures.

First, we introduce some definitions \cite{do1992riemannian, frankel2011geometry, gray1996modern} that we will use in the modeling and obtain some preliminary results.
%%%%%%%%%%%%%%%%%%%%%%%%%%%%%%%%%%%%%%%%%%%%
%\subsection{Definitions from differential geometry and preliminary results}
%In this section, we discuss some definitions in differential geometry \cite{do1992riemannian, frankel2011geometry, gray1996modern} and obtain important preliminary results.

\pagebreak
\begin{definition} (See \cite{do1992riemannian})
An immersion is a differentiable function between differentiable manifolds whose derivative is everywhere injective. Explicitly, $f : M \to N$ is an immersion if
$D_{p}f:T_{p}M\to T_{f(p)}N$ is an injective function at every point $p$ of $M$ (where $T_pX$ denotes the tangent space of a manifold $X$ at a point $p$ in $X$). 
\end{definition}

Let $\mathbf{R}:D \subset \mathbb{R}^3 \to \mathbb{R}^3$; $\mathbf{R}(u,v,w)=\left\langle X(u,v,w),Y(u,v,w),Z(u,v,w)\right \rangle$ represents an immersion of a three dimensional object $M$ in $\mathbb{R}^3$;
$M=R(D)$ where $D$ is an open simply connected domain.
\begin{definition}
The induced metric associated to $\mathbf{R}(u,v,w)$ is defined as
\begin{align}
G=\begin{bmatrix}G_{ij}\end{bmatrix}&=
\begin{bmatrix}
\langle\mathbf{R}_u , \mathbf{R}_u\rangle & \langle\mathbf{R}_u , \mathbf{R}_v\rangle & \langle\mathbf{R}_u , \mathbf{R}_w \rangle\\
\langle\mathbf{R}_v , \mathbf{R}_u\rangle & \langle\mathbf{R}_v , \mathbf{R}_v\rangle & \langle\mathbf{R}_v , \mathbf{R}_w \rangle\\
\langle\mathbf{R}_w , \mathbf{R}_u\rangle & \langle\mathbf{R}_w , \mathbf{R}_v\rangle & \langle\mathbf{R}_w , \mathbf{R}_w \rangle
\end{bmatrix}
\end{align}
 where $\langle \cdot , \cdot \rangle$ represents the inner product in $\mathbb{R}^3$.
\end{definition}

We will denote by $(M,G)$, the induced Riemannian manifold with metric $G$ on $M$. For the next definitions, please see \cite{frankel2011geometry}.

% \begin{definition}
% Second Fundamental Form\\
% The second fundamental form of $\mathbf{R}(u,v,w)$ is defined as
% \begin{align}
% =
% \begin{bmatrix}
% \langle\mathbf{R}_{uu} , \mathbf{n}\rangle & \langle\mathbf{R}_{uv} , \mathbf{n}\rangle & \langle\mathbf{R}_{uw} , \mathbf{n}\rangle\\
% \langle\mathbf{R}_{vu} , \mathbf{n}\rangle & \langle\mathbf{R}_{vv} , \mathbf{n}\rangle & \langle\mathbf{R}_{vw} , \mathbf{n}\rangle\\
% \langle\mathbf{R}_{wu} , \mathbf{n}\rangle & \langle\mathbf{R}_{wv} , \mathbf{n}\rangle & \langle\mathbf{R}_{ww} , \mathbf{n}\rangle\\
% \end{bmatrix}
% \end{align}
% \end{definition}

\begin{definition}
The norm of a vector $\mathbf{\bm \Phi}=\langle\phi^1,\phi^2,\phi^3\rangle$ on the manifold is defined by
\begin{align}
\|\bm{\Phi}\|_\subG =& \sqrt{\sum_{i,j}G_{ij} \phi^i \phi^j} = \sqrt{\langle G \bm{\Phi},\bm{\Phi} \rangle} .\label{norm}
\end{align}
\end{definition}

\begin{definition} 
Let $G^{-1}=G^{ij}$ be the inverse of G. The gradient of a differentiable function $\phi$ on the manifold is defined by
\begin{align} 
\nabla_\subG \phi &= \sum_{i,j}G^{ij}{\partial_i}\phi \,{\partial_j} \,. \label{grad} 
\end{align}
Note:\\
In local coordinates, the $j^{th}$ component of the vector field $\nabla_\subG \phi= \left( \nabla_\subG \phi \right)^j = \sum_{i}G^{ij}{\partial_i}\phi \\= \left( G^{-1} \nabla \phi \right)^j$ with corresponding basis $\left\{ \partial_j \right\}_{j=\{1,2,3\}} = \left\{ \frac{\partial}{\partial x_j}\right\}_{j=\{1,2,3\}} $

\end{definition}

Using Eqs.~\eqref{norm} and ~\eqref{grad}, we obtain the norm of the gradient on the manifold 	
	\begin{align} 
	\|\nabla_\subG \phi\|_\subG^2 &= \sum_{i,j}G^{ij}{\partial_i}\phi \,{\partial_j}\phi= \langle G^{-1} \nabla \phi,\nabla \phi \rangle. \label{grad_norm} 
	\end{align}
\begin{definition}
Let $|G|$ be the determinant of G. The divergence of a vector field $\bm{\Phi}$ on the manifold is defined by

	\begin{equation}\label{div}
	\nabla_\subG \cdot \bm{\Phi}= \sum_{i} \frac{1}{\sqrt{\left| G\right|}}{\partial_i}(\sqrt{\left| G\right|}\phi^i) \,.
	\end{equation}
\end{definition}

Using Eqs.~\eqref{grad} and \eqref{div}, the Laplace-Beltrami operator applied to a scalar $\phi$ is given by
\begin{align}
\Delta_\subG \phi = \nabla_\subG \cdot \nabla_\subG \phi = \frac{1}{\sqrt{\left| G\right|}}\sum_{i,j}{\partial_i}\left(G^{ij}\sqrt{\left| G \right|}{\partial_j}\phi\right) \label{Laplace_bel}\,.
\end{align}

Next, we obtain the Darcy-Forchheimer equation on the manifold $M$. 

\begin{lemma}
Let 
\begin{equation}
f_\beta \left(\|\nabla_\subG p\|_\subG\right)= \dfrac{2}{\alpha + \sqrt{\alpha^2 +4 \beta \|\nabla_\subG p\|_\subG}} \label{fbeta_man}, 
\end{equation}
where $p$ is the pressure.
 Then the velocity defined by the Darcy Forchheimer equation on the manifold M
\begin{equation}
\mathbf{v}= -f_\beta \left(\|\nabla_\subG  p\|_\subG\right) \nabla_\subG p\,, \label{DF_M}
\end{equation}
 solves the Forchheimer equation on M , 
 \begin{equation}
   \alpha \mathbf{v} + \beta \| \mathbf{v} \|_\subG \mathbf{v}=-\nabla_\subG p\,. \label{F_M}
  \end{equation}
\end{lemma}

\begin{proof}
Using Eq.~$\eqref{DF_M}$, Eq.~$\eqref{F_M}$ can be rewritten as
%	\begin{align}
%	- \alpha f_\beta \nabla_\subG p - \beta \sqrt{ \left(f_\beta \nabla_\subG p \right)\,G \left(f_\beta \nabla_\subG p\right)}f_\beta \nabla_\subG p= -\nabla_\subG p \,, 
%	\end{align}
%	or equivalently
	\begin{align}
	\left(\beta \sqrt{\left(\nabla_\subG p\right)G  \left(\nabla_\subG p\right)}  f_\beta^2 +\alpha f_\beta-1\right) \nabla_\subG p=\bm 0\,.
		\end{align}
	Using Eq.~\eqref{grad_norm}, we have $\sqrt{\left(\nabla_\subG p\right)G  \left(\nabla_\subG p\right)}= \| \nabla_\subG p \| _\subG$ and therefore
		\begin{align}
	\left(\beta \| \nabla_\subG p \|_\subG  f_\beta^2 +\alpha f_\beta-1\right) \nabla_\subG p=\bm 0\,.
		\end{align}
We need to show that \eqref{fbeta_man} solves the above equation. 
For $\nabla_\subG p = \bm 0$ it is true for any $f_\beta$.  
For $\|\nabla_\subG p\|_\subG\neq0 $, then $f_\beta$ should satisfy
	\begin{equation}
	\beta \| \nabla_\subG p \|_\subG  f_\beta^2 +\alpha f_\beta-1=0\,.
	\end{equation}
	The positive root is given by 
	\begin{equation}
	f_\beta(\|\nabla_\subG p\|_\subG) = \frac{-\alpha+\sqrt{\alpha^2+4\beta \|\nabla_\subG p\|_\subG}}{2\beta\|\nabla_\subG p\|_\subG},\label{aaa}
	\end{equation}
	or equivalently
	\begin{equation}
	f_\beta(\|\nabla_\subG p\|_\subG)= \frac{2}{\alpha+\sqrt {\alpha^2 +4 \beta \|\nabla_\subG p\|_\subG}},
	\end{equation}
	which is obtained multiplying both denominator and numerator of\eqref{aaa} by $$\alpha+\sqrt{\alpha^2+4\beta \|\nabla_\subG p\|_\subG}.$$
	
\end{proof}

%%%%%%%%%%%%%%%%%%%%%%%%%%%%%%%
%%%%%%%%%%%%%%%%%%%%%%%%%%%%%%%%%%%%%%%%%%%%
\subsection{Geometric model of the fracture as a 3-D manifold} \label{Hyp_Surf}
In this section, we model the fracture as a 3-D manifold immersed in porous media.
In our approach, we formulate the fracture as a parametrized Riemannian hypersurface.
(Please see \cite{paranamana2018analytical} for more details.) We describe the fracture as the normal variation of a surface $\mathbf{r}(u,v)$ immersed in $\mathbb{R}^3$ given by, 
\begin{align}
\mathbf{R}(u,v,\lambda)=\mathbf r (u, v) + \lambda \, \mathbf{n}(u, v),
\end{align} 
 where 
 \begin{align}
 \mathbf{r}(u,v)=\left \langle X(u,v),Y(u,v),Z(u,v)\right \rangle,
\end{align} $\mathbf{n}$ is the outward unit normal vector to the surface such that 
\begin{align}
\mathbf{n} = \dfrac{\mathbf{r}_u \times \mathbf{r}_v}{\|\mathbf{r}_u \times \mathbf{r}_v\|} \label{normal}
\end{align} 
and $\lambda\in[-h(u,v),h(u,v)]$. Here $2h(u,v)$ represents the thickness of the fracture. 
\begin{figure}[h]
\hspace{1cm}
\includegraphics[scale=0.4]{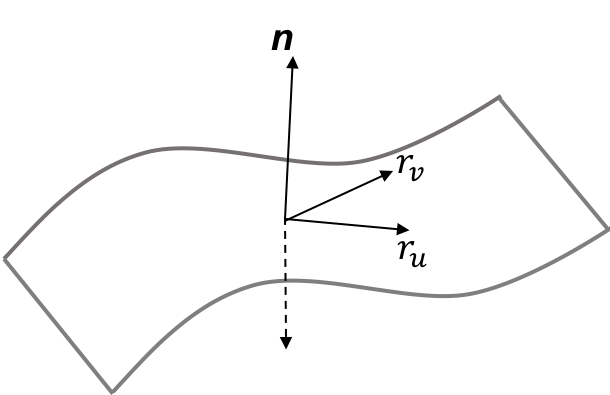}
\hspace{1cm}
\includegraphics[scale=0.5]{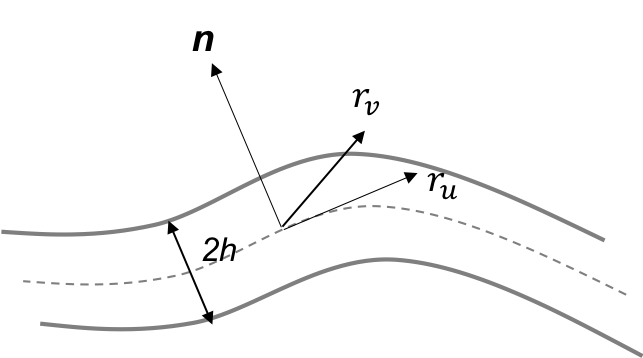}  
\end{figure}
%Then \begin{align}
%\mathbf{R}_u(u,v,\lambda)=\mathbf{r}_u(u,v)+ \lambda  \mathbf{n}_u(u,v)\,,\\
%\mathbf{R}_v(u,v,\lambda)=\mathbf{r}_v(u,v)+ \lambda  \mathbf{n}_v(u,v) \,,\\
%\mathbf{R}_\lambda(u,v,\lambda)=\mathbf{n}(u,v) \,.
%\end{align}

Let $g=g_{ij}$ be the first fundamental form of $\mathbf{r}(u,v)$
 given by
\begin{align}
g=g_{ij}=
\begin{bmatrix}
\langle\mathbf{r}_u , \mathbf{r}_u\rangle & \langle\mathbf{r}_u , \mathbf{r}_v\rangle\\
\langle\mathbf{r}_v , \mathbf{r}_u\rangle & \langle\mathbf{r}_v , \mathbf{r}_v\rangle
\end{bmatrix}.
\end{align}
%Then the coefficients of the first fundamental form of $\mathbf{R}(u,v,\lambda)$ are
%\begin{align}
%G_{11}=& g_{11}+ \lambda^2 |\mathbf{n}_u(u,v)|^2 + 2\lambda \langle\mathbf{r}_u(u,v),\mathbf{n}_u(u,v)\rangle ,\\
%G_{22}=& g_{22}+\lambda^2  |\mathbf{n}_v(u,v)|^2 + 2 \lambda \langle \mathbf{r}_v(u,v),\mathbf{n}_v(u,v)\rangle,  \\
%G_{12}= &G_{21}= g_{12} + \lambda^2 \langle\mathbf{n}_u(u,v),\mathbf{n}_v(u,v)\rangle \nonumber \\
%& +\lambda (\langle \mathbf{r}_u(u,v),\mathbf{n}_v(u,v) \rangle +\langle\mathbf{r}_v(u,v),\mathbf{n}_u(u,v)\rangle),\\
%G_{13}=& G_{31}=\langle\mathbf{r}_u(u,v),\mathbf{n}(u,v)\rangle + \lambda \langle\mathbf{n}_u(u,v),\mathbf{n}(u,v)\rangle,\\
%G_{23}=& G_{32}= \langle\mathbf{r}_v(u,v),\mathbf{n}(u,v)\rangle + \lambda \langle\mathbf{n}_v(u,v),\mathbf{n}(u,v)\rangle ,\\
%G_{33}=& \langle\mathbf{n}(u,v),\mathbf{n}(u,v)\rangle.
%\end{align}
and second fundamental form  of $\mathbf{r}(u,v)$ is given by
\begin{align}
\begin{bmatrix}
 l & m \\
 m & n 
\end{bmatrix}
=
\begin{bmatrix}
\langle\mathbf{r}_{uu} , \mathbf{n}\rangle & \langle\mathbf{r}_{uv} , \mathbf{n}\rangle\\
\langle\mathbf{r}_{vu} , \mathbf{n}\rangle & \langle\mathbf{r}_{vv} , \mathbf{n}\rangle 
\end{bmatrix}.
\end{align}
%Since $\langle\mathbf{r}_i(u,v),\mathbf{n}(u,v)\rangle=0$, 
%we have 
%\begin{align}
%\langle\mathbf{r}_{ij}(u,v),\mathbf{n}(u,v)\rangle+\langle\mathbf{r}_{i}(u,v),\mathbf{n}_j(u,v)\rangle=0
%\end{align} where $i,j=\{u,v\}$. Therefore
%\begin{align}
%l&=-\langle \mathbf{r}_u(u,v),\mathbf{n}_u(u,v) \rangle \,\\
%m&=-\langle \mathbf{r}_u(u,v),\mathbf{n}_v(u,v) \rangle=-\langle \mathbf{r}_v(u,v),\mathbf{n}_u(u,v) \rangle \, \\
%n&= -\langle \mathbf{r}_v(u,v),\mathbf{n}_v(u,v) \rangle=-\langle \mathbf{r}_v(u,v),\mathbf{n}_v(u,v) \rangle \,.
%\end{align}
Since $$\langle\mathbf{r}_i(u,v),\mathbf{n}(u,v)\rangle=0,\, 
\langle\mathbf{n}_i(u,v,),\mathbf{n}(u,v)\rangle=0\,,
\langle\mathbf{n}(u,v),\mathbf{n}(u,v)\rangle=1,  
$$
for ${i,j}=\{u,v\}$,
the coefficients of the first fundamental form can be rewritten as
$$
G_{11}= g_{11} -2 l\lambda  +  |\mathbf{n}_u|^2 \lambda^2 ,\,\, \,
G_{12}= g_{12} -2m \lambda +  \langle\mathbf{n}_u,\mathbf{n}_v\rangle \lambda^2  ,\,\,\,
G_{22}= g_{22}  - 2 n \lambda  + |\mathbf{n}_v|^2 \lambda^2,$$  $$  
G_{13}= G_{31}=0,\, \,\,
G_{23}= G_{32}=0,\, \,\,
G_{33}= 1.
$$
Let $|g|$ be the determinant of $g$, $K$ be the Gaussian curvature of $\mathbf{r}(u,v)$, and  $H$ be the mean curvature of $\mathbf{r}(u,v)$ given by \cite{do1992riemannian}
$$
|g|=g_{11} g_{22}-g_{12}^2,\,\,\,
K=\frac{ln-m^2}{|g|},\,\,\,
H=\frac{g_{11}\,n - 2 g_{12}\,m + g_{22}\,l}{2|g|}.$$
Let $|G|$ be the determinant and $G^{-1}=G^{ij}$ be the inverse matrix of $G$, respectively. 
We have 
\begin{align}
|G|
=&\;|g|-4H|g|\lambda
+ \Big(4 K |g| + g_{11} |\mathbf{n}_v|^2 -2 g_{12} \langle\mathbf{n}_u,\mathbf{n}_v\rangle 
+g_{22} |\mathbf{n}_u |^2 \Big) \lambda^2 \nonumber \\
&-2\Big(l |\mathbf{n}_v|^2 -2m \langle\mathbf{n}_u,\mathbf{n}_v\rangle 
+n |\mathbf{n}_u|^2 \Big) \lambda^3 
+\Big( |\mathbf{n}_u|^2 |\mathbf{n}_v|^2 - \langle\mathbf{n}_u,\mathbf{n}_v\rangle^2
 \Big) \lambda^4\,. \label{detG}
\end{align}
%Let $G^{-1}=G^{ij}$ be the inverse matrix of $G$. %Then
%\begin{align}
%G^{-1}=
%\; \frac{1}{|G|}\left(
%\begin{array}{ccc}
%G_{22}   & -G_{12} & 0 \\
%-G_{12} & G_{11} & 0 \\
%0 & 0 & |G| \\
%\end{array}
%\right).
%\end{align}
With this metric $G$ and the coefficients defined above, next, we obtain the equation for pressure of the flow inside the fracture domain. 
\subsection{Flow equation inside the fracture}
Now we model the pressure distribution of a nonlinear flow inside the fracture domain $\Omega_f,$ which is defined as  a general manifold described above. The boundary of the domain of the flow is split as $\partial{\Omega_f}=\Gamma_w \cup \Gamma_f^{\pm} \cup \Gamma_{out}^\pm \cup \Gamma_{f_{out}}.$ 
\pagebreak
\begin{align*}
\Omega_f&=\{(u,v,\lambda):0<u<L, a<v<b, -h(u,v)<\lambda<h(u,v)\}\,,\\
\Gamma_w&=\{(u,v,\lambda):u=0, a<v<b, -h(0,v)<\lambda<h(0,v)\}\,,\\
\Gamma_f^\pm&= \{(u,v,\lambda):0<u<L, a<v<b, \lambda=\pm h(u,v)\}\,,\\
\Gamma_{f_{out}}&=\{(u,v,\lambda):u=L, a<v<b, -h(L,v)<\lambda<h(L,v)\}\,,\\
\Gamma_{out}^+&=\{(u,v,\lambda):v=b, 0<u<L, -h(u,b)<\lambda<h(u,b)\,,\\
\Gamma_{out}^-&=\{(u,v,\lambda):v=a, 0<u<L, -h(u,a)<\lambda<h(u,a)\,.
\end{align*}  
Schematically the domain of the flow with its boundaries is presented in the Figure \ref{manifold}.
\begin{figure}[h!]
\begin{center}
\includegraphics[scale=0.15]{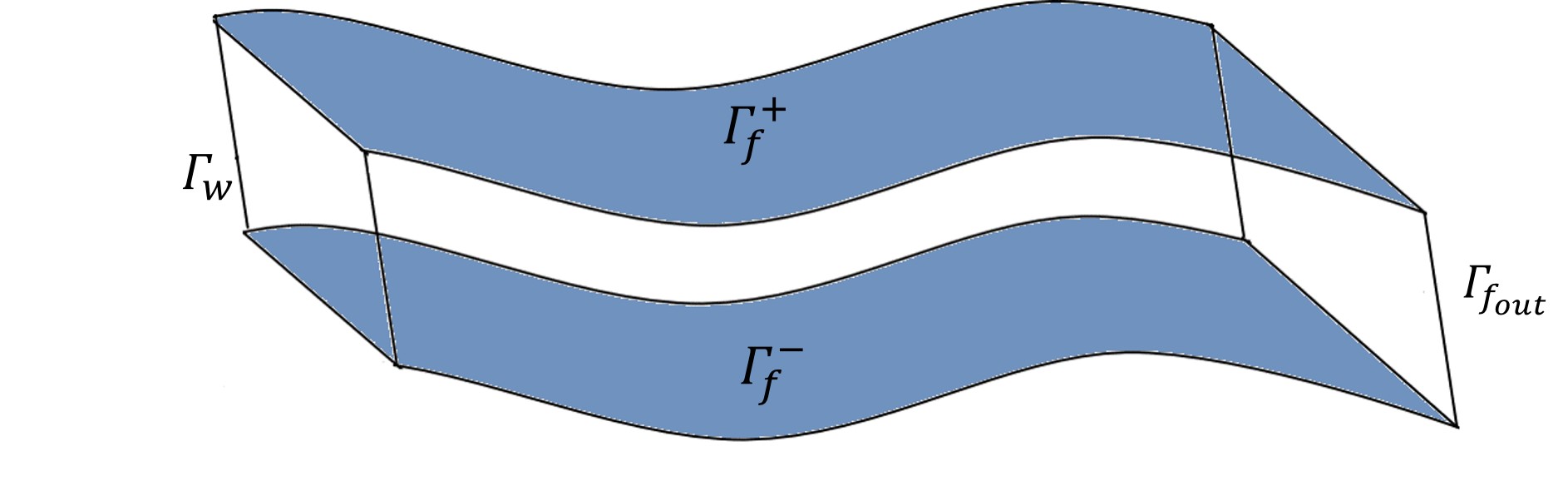}
\caption{Schematic of the fracture domain $\Omega_f$ as a manifold.}\label{manifold}
\end{center} \vspace{-3mm}
\end{figure}

Using Eqs.~\eqref{frac pressure},~\eqref{grad_norm} and \eqref{Laplace_bel} we have the pressure of the flow inside $\Omega_f$ given by the equation
\begin{align}
-\frac{1}{\sqrt{ |G| }}\sum_{i,j}{\partial_i}\left(f_\beta \left(\|\nabla_\subG W\|_\subG\right)\sqrt{{|G|}}\, G^{ij}{\partial_j}W\right) = \frac{Q}{\left|\Omega\right|}\  \  \text{in} \ \Omega_f\,, \label{actual}
\end{align}
or equivalently 
\begin{align}
-\sum_{i,j}{\partial_i}\left(f_\beta \left(\|\nabla_\subG W\|_\subG\right) \sqrt{{|G|}}\, G^{ij}{\partial_j}W\right) = 
\sqrt{ |G|} \frac{Q}{\left|\Omega\right|} \  \text{in} \ \Omega_f\,,\label{actual1}
\end{align}
where	
$$
f_\beta \left(\|\nabla_\subG W\|_\subG\right)= \frac{2}{\alpha + \sqrt{\alpha^2 +4 \beta \|\nabla_\subG W\|_\subG}} ,\,\,
\|\nabla_\subG W\|^2_\subG= \sum_{i,j}G^{ij}{\partial_i}W \,{\partial_j}W ,\,
$$
 $i,j=\{1,2,3 \} \,\, \text{and} \,\, \partial_1 =\partial_u, \partial_2 =\partial_v \,,\,\partial_3 =\partial_{\lambda}.$
Let $\mathbf{n}_f^{\pm}$ be the unit outward normal vector to the top/bottom boundary $\Gamma_f^{\pm}$ and $\mathbf{n}$ be the unit outward normal vector to the boundaries $\Gamma_{f_{out}}$ and $\Gamma_{out}^\pm$.
Following mixed boundary conditions are imposed on the boundaries: Dirichlet condition on the well $\Gamma_w$, flux conditions on the top and the bottom boundaries $ \Gamma_f^\pm$, no flux condition on outer boundaries $\Gamma_{f_{out}}$ and $\Gamma_{out}^\pm .$ Moreover, for simplicity we assume no flow in $v$ direction.  
Namely,
\begin{align}
W&=0 && \mbox{ on } \Gamma_w\,,\label{well}\\
\mathbf{v} \cdot \mathbf{n}_f^{\pm}&=-q^{\pm}(u,v)&& \mbox{ on } \Gamma_f^\pm\,,\\
\mathbf{v} \cdot \mathbf{n}&=0&& \mbox{ on }  \Gamma_{f_{out}} \cup \Gamma_{out}^\pm \label{noFluxBdry}\,.
\end{align}

\subsection{Reduced model of the flow inside the fracture}\label{ModelReduction}
In this section we simplify the original flow equation inside the fracture (given by Eq.~\eqref{actual1}) and obtain a reduced dimensional model for the flow in the fracture. 

Using the fact that $G^{23}=G^{32}=0,\,G^{33}=1$ and integrating Eq.~\eqref{actual1} over the thickness of the fracture, we get
\pagebreak
\begin{align} 
-\frac{Q}{\left|\Omega\right|} \int_{-h(u,v)}^{h(u,v)} \sqrt{ |G|}   d\lambda
=& \int_{-h(u,v)}^{h(u,v)} \Big [ \sum_{i,j}
\partial_i\left( f_\beta \left(\|\nabla_\subG W\|_\subG\right) \sqrt{{|G|}}\, G^{ij}\,\partial_j W  \right) \nonumber \\
& + {\partial_{\lambda}} \left( f_\beta \left(\|\nabla_\subG W\|_\subG\right)\sqrt{{|G|}}\,\partial_\lambda W \right) \Big ]\,\, d \lambda \,, \label{pressure_on_frac}
\end{align}
where $i,j=\{1,2\}$.

\begin{proposition} \label{red_prop}
 Assume that the boundary conditions \eqref{well}-\eqref{noFluxBdry} hold. Then Eq.~\eqref{pressure_on_frac} can be reduced to 
\begin{align} 
-\sum_{i,j} \partial_i  \int_{-h(u,v)}^{h(u,v)}
 &\left( f_\beta \left(\|\nabla_\subG W\|_\subG\right) \sqrt{{|G|}}\, G^{ij}\,\partial_j W \right) d \lambda 
 \nonumber \\
 &\qquad =\frac{Q}{\left|\Omega\right|} \int_{-h(u,v)}^{h(u,v)} \sqrt{ |G|}   d\lambda
+ \tilde{q}^+(u,v) + \tilde{q}^-(u,v) \,. \label{red_eq}
\end{align}
where $i,j=\{1,2\}$,
$$
\tilde{q}^+(u,v)=q^+\sqrt{1+h_u^2+h_v^2}\sqrt{|G|} ,\,\,
\tilde{q}^-(u,v)=q^-\sqrt{1+h_u^2+h_v^2}\sqrt{|G|}.$$
\end{proposition}
\begin{proof}
Using the Leibniz rule, the first integral of the right hand side of Eq.~\eqref{pressure_on_frac} can be rewritten as
\begin{align}
\nonumber \int_{-h}^{h} 
 \sum_{i,j}& \partial_i \big( f_\beta  \left(\|\nabla_\subG W\|_\subG\right)  \sqrt{{|G|}}\, G^{ij}\,\partial_j W \big) d \lambda = \nonumber \\
&\sum_{i,j}\Big[\partial_i  \int_{-h}^{h}
 \left( f_\beta \left(\|\nabla_\subG W\|_\subG\right) \sqrt{{|G|}}\, G^{ij}\,\partial_j W \right)  d \lambda \nonumber \\
 &\qquad  - h_i \Big(  \left( f_\beta \left(\|\nabla_\subG W\|_\subG \right) \sqrt{{|G|}}\, G^{ij}\,\partial_j W \right) \Big \rvert_{h} \nonumber \\
 &\qquad \qquad +\left( f_\beta \left(\|\nabla_\subG W\|_\subG\right) \sqrt{{|G|}}\, G^{ij}\,\partial_j W \right) \Big \rvert_{-h} \Big)\Big] \label{int_1} \,,
\end{align}
where $h_1=h_u$ and $h_2=h_v$,
and the second integral can be rewritten as
\begin{align}
\int_{-{h}}^{h} {\partial_{\lambda}} \left( f_\beta \left(\|\nabla_\subG W\|\right) \sqrt{{|G|}}\,\partial_\lambda W \right) \, d \lambda 
=& \left( f_\beta \left(\|\nabla_\subG W\|_\subG\right) \sqrt{{|G|}}\,\partial_\lambda W  \right) \Big \rvert_{\, h} \nonumber\\
-& \left( f_\beta \left(\|\nabla_\subG W\|_\subG\right) \sqrt{{|G|}}\,\partial_\lambda W  \right) \Big \rvert_{\,-h}. \label{int_2}
\end{align}
%Using Eqs.~\eqref{int_1} and \eqref{int_2} we have
%\begin{align}
% -h_i& \Big(  \left( f_\beta  \sqrt{{|G|}}\, G^{ij}\,\partial_j W \right) \Big \rvert_{h} +\left( f_\beta \sqrt{{|G|}}\, G^{ij}\,\partial_j W \right) \Big \rvert_{-h} \Big)\nonumber \\
% &+\left( f_\beta  \sqrt{{|G|}}\,\partial_\lambda W  \right) \Big \rvert_{ h} - \left( f_\beta \sqrt{{|G|}}\,\partial_\lambda W  \right) \Big \rvert_{-h}\nonumber\\
%& = \Big[f_\beta \sqrt{|G|}  \Big(\left(-h_u G^{11}-h_v G^{21}\right)\partial_u W +\left(-h_u G^{12}-h_v G^{22}\right)\partial_v W + \partial_ \lambda W \Big)\Big] \bigg \rvert_{h} \nonumber \\
%&  + \Big[ f_\beta \sqrt{|G|}\Big(\left(-h_u G^{11}-h_v G^{21}\right)\partial_u W +\left(-h_u G^{12}-h_v G^{22}\right)\partial_v W - \partial_ \lambda W \Big)\Big] \bigg \rvert_{-h}
% \,. \label{89}
%\end{align}
Due to Eqs.~ \eqref{grad}, \eqref{DF_M} and \eqref{normal}
\begin{align}
\mathbf{n}_f^+= \frac{\langle -h_u,-h_v,1\rangle}{\sqrt{1+h_u^2+h_v^2}} \,,
\end{align}
\begin{align}
 \mathbf{v} = -f_\beta \nabla_\subG W  = -f_\beta \sum_{i,j} G^{ij} \partial_i W \partial_j \,.
 \end{align}
 
 \pagebreak
Therefore, on $\Gamma^+_f$ we have,
\begin{align}
&\Big[f_\beta \sqrt{|G|}  \Big(\left(-h_u G^{11}-h_v G^{21}\right)\partial_u W +\left(-h_u G^{12}-h_v G^{22}\right)\partial_v W + \partial_ \lambda W \Big)\Big] \bigg \rvert_{h} \nonumber \\
& =  f_\beta \sqrt{|G|} \langle G^{11}\partial_u W +G^{12}\partial_v W, G^{21}\partial_u W +G^{22}\partial_v W, \partial_\lambda W\rangle \cdot \langle -h_u,-h_v,1\rangle ^{T}
\nonumber \\
&= \sqrt{|G|}\,\,f_\beta \nabla_\subG W \cdot \mathbf{n}_f^+\sqrt{1+h_u^2+h_v^2}\nonumber \\
%&= -\sqrt{1+h_u^2+h_v^2}\sqrt{|G|}\,(\mathbf{v}\cdot \mathbf{n}_f^+)\nonumber \\
%&= \sqrt{1+h_u^2+h_v^2}\sqrt{|G|}\,\, q^+ \nonumber \\
&=\tilde{q}^+(u,v) \, \label{flux+}.
\end{align}
Similarly we can show that on $\Gamma^-_f$,
\begin{align}
\Big[f_\beta \sqrt{|G|}  &\Big(\left(-h_u G^{11}-h_v G^{21}\right)\partial_u W +\left(-h_u G^{12}-h_v G^{22}\right)\partial_v W + \partial_ \lambda W \Big)\Big] \bigg 
\rvert_{-h} \nonumber \\
&= \tilde{q}^-(u,v) \, \label{flux-}.
\end{align}
%Substituting Eqs.~\eqref{flux+} and~\eqref{flux-} to Eq.~\eqref{89}, we can obtain the following.
Hence,
\begin{align}
-h_i \Big(  \big( f_\beta  & \sqrt{{|G|}}\, G^{ij}\,\partial_j W \big) \Big \rvert_{h} +\big( f_\beta \sqrt{{|G|}}\, G^{ij}\,\partial_j W \big) \Big \rvert_{-h} \Big)\nonumber \\
 +&\big( f_\beta  \sqrt{{|G|}}\,\partial_\lambda W  \big) \Big \rvert_{ h} - \big( f_\beta \sqrt{{|G|}}\,\partial_\lambda W  \big) \Big \rvert_{-h}= \,\,\tilde{q}^+(u,v) + \tilde{q}^-(u,v) \label{q+,q-}
\end{align}
 Therefore, using Eqs.\eqref{q+,q-},~\eqref{int_1} and \eqref{int_2}, Eq.~\eqref{pressure_on_frac} can be rewritten as,
 \begin{align} 
-\sum_{i,j} \partial_i  \int_{-h(u,v)}^{h(u,v)}
 &\left( f_\beta \left(\|\nabla_\subG W\|_\subG\right) \sqrt{{|G|}}\, G^{ij}\,\partial_j W \right) d \lambda 
 \nonumber \\
 &\qquad =\frac{Q}{\left|\Omega\right|} \int_{-h(u,v)}^{h(u,v)} \sqrt{ |G|}   d\lambda
+ \tilde{q}^+(u,v) + \tilde{q}^-(u,v) \,. 
\end{align}
\end{proof}

The thickness of the fracture is several orders smaller compared to the length of the fracture. Therefore it is reasonable to assume that the flow inside the fracture in the normal direction to the barycentric surface is negligible. With that assumption, we obtain the reduced model for the flow pressure in the fracture.

\begin{proposition}\label{FINAL_THM}
\textbf{Reduced Model I.} 
Let all the conditions of proposition \ref{red_prop} are satisfied. Assume that the gradient of $W$ is independent of $\lambda.$ Then the equation for pressure of the flow inside the fracture can be given by,
\begin{align}\label{FINAL}
 -\sum_{i,j}{\partial_i}(L^{ij}  {\partial_j}W)=&
\frac{Q}{\left|\Omega\right|}A + \tilde{q}^+(u,v)+\tilde{q}^-(u,v) \,, 
\end{align} 
where
$$
L^{ij}= \int_{-h}^h
  \left( f_\beta \left(\|\nabla_\subG W\|_\subG\right) \sqrt{{|G|}}\, G^{ij} \right)  d \lambda,\,\,
 A= \int_{-h}^h \sqrt{{|G|}} d \lambda,$$
$ i,j=\{1,2\}$
\end{proposition}

We use numerical integration (3-point Gauss Quadrature Rule) to evaluate integrals $L^{ij}$ for $i,j \in \{1,2\} $ and $A$.
%%%%%%%%%%%%%%%%%%%%%%%%%%%%%%%%%%%%%%%%%%%%%%%%%%%%%%%%%

\subsection{Reduced model of the fracture: when the metric $G$ does not depend on $\lambda$}

For fractures with small thicknesses it is reasonable to assume that the solution of the flow equation in the direction of the thickness does not change, namely, we assume that the solution $W$ does not depend on the parameter $\lambda$ and  consequently the first fundamental form of $\mathbf{R}(u,v,\lambda)$, $G$, does not depend on $\lambda$ as well.  Therefore, next we further simplify the Reduced Model I. 
%
%Assume that the first fundamental form of $\mathbf{R}(u,v,\lambda)$ is $\lambda$ independent.
%

\begin{proposition}\label{red_Thm}
\textbf{Reduced Model II.} Let all the conditions of theorem \ref{FINAL_THM} are satisfied.
%Assume that $W$ does not change in the normal direction. That is $W$ is assumed to be $\lambda$ independent.
%Let $\mathbf{n}^\pm$ be the unit outward normal vector to the top/bottom boundary $\Gamma^\pm$.
%Assume that the flux on $\Gamma^\pm$ is $ \mathbf{v} \cdot \mathbf{n}^\pm=-q^\pm$.
In addition assume that the first fundamental form of $\mathbf{R}(u,v,\lambda)$ depends on the function $\mathbf{r}(u,v)$ only (does not depend on $\lambda$):
\begin{align}
G=	 G(u,v,0)=
\begin{bmatrix}
g_{11} & g_{12} & 0\\
g_{21} & g_{22} & 0\\
0 & 0 & 1
\end{bmatrix}.
\end{align} 
Then the pressure inside the fracture is subjected to the equation:
\begin{align}
 - 2 \sum_{i,j}{\partial_i}
   \left( h(u,v)f_\beta \left(\|\nabla_g W\|_g \right) \sqrt{{|g|}}\, g^{ij}    {\partial_j}W \right) =&
2 \frac{Q}{\left|\Omega\right|} h(u,v) \sqrt{{|g|}} \nonumber \\
&+\tilde{q}^+(u,v) +\tilde{q}^-(u,v) \,,\label{FINAL2} 
\end{align}
$i,j=\{1,2\}\,, $
where,
$$
\tilde{q}^+(u,v)=q^+\sqrt{1+h_u^2+h_v^2}\sqrt{|g|},\,\,\quad
\tilde{q}^-(u,v)=q^-\sqrt{1+h_u^2+h_v^2}\sqrt{|g|} 
$$
\end{proposition}
Proof follows from theorem \ref{FINAL_THM}.
%\begin{proof}
%Since $G$ does not depend on $\lambda$, from theorem \ref{FINAL_THM}, we obtain  the equation for flow
%\begin{align}
% -\sum_{i,j}{\partial_i}(L^{ij}  {\partial_j}W)=&
%\frac{Q}{\left|\Omega\right|}A + \tilde{q}^+(u,v)+\tilde{q}^-(u,v) \,.
%\end{align}
%with
%\begin{align}
% L^{ij}&= 2h(u,v)  f_\beta \left(\|\nabla_g W\|_g\right) \sqrt{{|g|}}\, g^{ij} \,,\\
% A&= 2 h(u,v) \sqrt{{|g|}}\,.
% \end{align}
%\end{proof}
\begin{remark}\label{remark_lambda}
The reason why we consider two reduced models is the following.
Reduced Model I is more comprehensive. It originates from the actual model under the assumption that the gradient of the pressure function $W$  does not depend on the parameter $\lambda$ which physically means that  the velocity inside the fracture in the orthogonal direction to the barycentric surface is negligible. In this case, the coefficients of the equation $L^{ij}$ implicitly depend on $\lambda$ through integration.  
%For Reduced Model I, we  perform all numerical experiments and  computationally prove its closeness to the actual model, for values of the fracture thickness $h\leq h_0$ for $h_0$ which can be \textit{relatively big}.  

Reduced Model II is a simplified version of the Reduced Model I, in which coefficients of the equation do not depend on $\lambda$. Consequently, the corresponding solution does not depend on $\lambda$. It is clear that, as the thickness of the fracture becomes big enough, the solution obtained from the Reduced Model II will significantly deviate from the actual model. In section \ref{NA}, we investigate this numerically in detail.
\end{remark} 

\section{Estimates for the difference between the solutions of the original model and the reduced models} \label{f}
%By Monge's theorem, locally, every surface is a graph. 
%We simplify the Reduced Model II when the fracture is treated as the normal variation of a graph and perform theoretical analysis. 
In this section, we provide estimates for the difference between the solution of the original model and the solution of the Reduced Model II, 
when the fracture domain $\Omega_f$ is considered to be a foliation of a cylindrical surface in $\mathbb{R}^3$, namely
$$\mathbf{r}(u,v) = \left \langle u,v,f(u)\right \rangle, \mbox{ and }  h(u,v) = h(u).$$
We first investigate the difference between the solutions only inside the fracture with given fluxes, Theorem \ref{Tdifference}, and then we investigate the difference between 
the solutions in the coupled domain, Theorem \ref{TdifferenceCoupled}.

\begin{proposition}\label{red_prop_g2}
Let all the conditions of proposition \ref{red_Thm}. Then the equation for pressure inside the fracture as a foliation of a cylindrical surface in $\mathbb{R}^3$is given by
\begin{align} 
- \partial_u \left( \frac{2h(u)}{\sqrt{ 1+f_u^2}}\,f_\beta \left(\|\nabla_g W\|_g \right) \, \partial_u W\right)=2\frac{Q}{\left|\Omega\right|} h(u)\sqrt{ 1+f_u^2}  + \tilde{q}^+(u) + \tilde{q}^-(u)\,
, \label{red_eq_g2}
\end{align}
where 
$$
\tilde{q}^\pm (u) =q^\pm \sqrt{1+h_u^2}\sqrt{1+f_u^2} ,\,\,\,
\|\nabla_g W \|_g = \frac{1}{\sqrt{1+f_u^2}}\partial_u W \,.
$$
\end{proposition}
\begin{proof}
For a fracture with $z(u,v)=f(u)$ and thickness $2h(u)$, the barycentric surface of the fracture is given by $\mathbf{r}(u,v) = \left \langle u,v,f(u)\right \rangle$ and the inverse of the induced metric associated to $\mathbf{R}(u,v,\lambda)$ is given by
\begin{align}
G^{-1}= 
\begin{bmatrix}
\dfrac{1}{1+f_u^2} & 0 & 0\\
0 & 1 & 0\\
0 & 0 & 1
\end{bmatrix}.
\end{align}
Then the result follows from proposition \ref{red_Thm}.
\end{proof}

The analysis will be based on the following results.

 \begin{theorem} \label{IBP} (See Ref.\cite{do2012differential})\\
Let $M$ be a compact oriented Riemannian manifold of dimension  $n$ with boundary $\partial M$.
Then for all the vector fields $X$ and smooth functions $f$, the integration by parts formula on the manifold is given by
\begin{equation} 
\int_M f \rm{div} X \, d\mu
 =- \int_M \langle \nabla_g f, X\rangle_g d \mu 
  + \int_{\partial M}f\cdot \langle X,\mathbf{n} \rangle_g d \tilde{\mu}\,.
\end{equation}
\end{theorem}
\begin{definition}
On the manifold $M$, $L^p$ norm is defined as the following  .
$$\|F\|_{L^p(M)}=\Big(\int_M \left|F\right|^p \sqrt{|g|} dudvd\lambda \Big)^{1/p} $$
\end{definition} 

\begin{lemma} \label{lidia_n}
For $f_{\beta}(\|\nabla_g W \|_g)$ defined by Eq.~\eqref{non lin term}, $1 \leq q < 2$,
\begin{align}
&\int_\Omega \Big( f_\beta \left(\|\nabla_g W _1\|_g \right) \nabla_g W_1-  f_\beta \left(\|\nabla_g W _2\|_g \right) \nabla_g  W_2 \Big) 
\cdot \nabla_g \left(W_1- W_2\right)   d\Omega
 \nonumber \\
 &  \geq C \left\|\nabla_g \left(W_1- W_2\right) \right\|_{L^q} ^2
 \Big \{1+ \max 
 \left( \|\nabla_g W_1\|_{L^{\frac{q}{2(2-q)}}}, \|\nabla_g W_2\|_{L^{\frac{q}{2(2-q)}}}\right)\Big \}^{-1/2}\,.
\end{align}
\end{lemma}
%\begin{proof}
The proof of the above lemma can be obtained using the same arguments as for the case of $\mathbb{R}^n$ in  Lemma III.11 in Ref \cite{aulisa2009analysis} with $a=1/2$.
%\end{proof}

\begin{lemma} \label{boundness_n}
There exists a constant $C$ depending on $\Omega$, $Q$, $q^+(u)$ and $q^-(u)$ such that the corresponding basic profiles $W$ and $\bar W$ satisfy
$$\|\nabla_g W\|_{L^\frac{3}{2}(\Omega)} \leq C, \qquad \|\nabla_g \bar W\|_{L^\frac{3}{2}(\Omega)} \leq C.$$
\end{lemma}
%\begin{proof}
The proof of the above lemma can be obtained using the same arguments as for the case of $\mathbb{R}^n$ in Theorem V.4 in Ref \cite{aulisa2009analysis} with $a=1/2$.  
%\end{proof}

\subsection{Estimates for the difference between the solutions inside the fracture $\Omega_f$.} 

\begin{figure}[H]
\begin{center}
\frame{\includegraphics[scale=0.4]{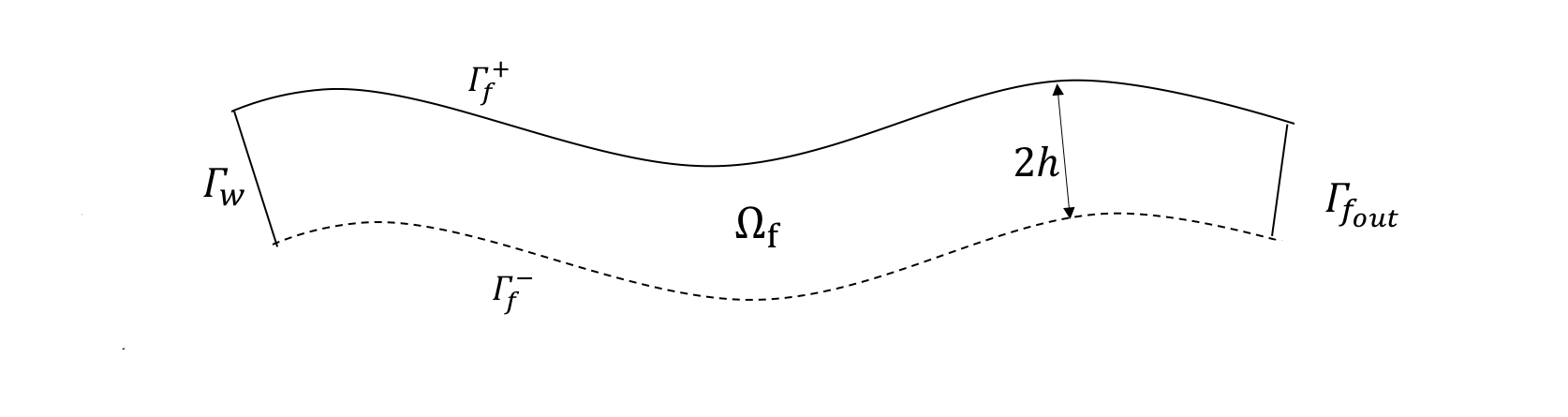}} 
\caption{Fracture domain $\Omega_f$.} 
\end{center} 
\end{figure}

We investigate the difference between the solutions of the actual model \eqref{prob_1} and the reduced model \eqref{prob_2} on the 2-D manifold as defined below. Let $W$ be the solution of the actual model and $\bar W$ be the solution of the reduced model. Here $\nabla_g = <\frac{1}{1+f_u^2}\partial_u,\partial_{\lambda}>$ and $\nabla_{g_u} = \frac{1}{1+f_u^2}\partial_u.$
\begin{enumerate}[(i)]
\item \label{prob_1} \textbf{Actual model}: 
The flow equation  is given by
\begin{equation}
-\nabla_g \cdot f_\beta(\|\nabla_g W\|_g)\nabla_g W= \frac{Q}{|\Omega|} \,\, \mbox{ in } \Omega_f  \,, \label{new1Ln}
\end{equation}
with the boundary conditions
\begin{align}
\left<f_\beta (\|\nabla_g W\|_g)\nabla_g W , \mathbf{n}_f^\pm\right>_g &= q^\pm(u), &&\mbox{ on } \Gamma_f^\pm, \label{bc1Ln} \\
W  &=0, &&\mbox{ on } \Gamma_w,\label{bc3Ln}\\
\left<f_\beta (\|\nabla_g W\|_g)\nabla_gW,\mathbf{n}\right>_g&=0
, &&\mbox{ on } \Gamma_{f_{out}}\label{bc4Ln}\,.
\end{align}
\item \label{prob_2}
%The flow equation for the reduced 1-D fracture model has the form:
% \begin{equation}
%-\nabla_{g_u} \cdot f_\beta(\|\nabla_{g_u}\bar W\|_{g_u})\nabla_{g_u} \bar W  
%=\frac{Q}{\left|\Omega\right|} + \frac{1}{2h} \left(q^+(u)+q^-(u) \right)  \mbox{ in } \ \  0<u<L  \,, \label{new2L_0}
%\end{equation}
%with boundary conditions
%\begin{align}
%\bar W  &=0 ,& &\mbox{ on } \Gamma_w,\label{bc2aL_0}\\
%\left<f_\beta (\|\nabla_{g_u} \bar W\|_{g_u})\nabla_{g_u} \bar W , \mathbf{n}\right>_{g_u} &=0,& &\mbox{ on } \Gamma_{f_{out}}\label{bc3aL_0} \,.
%\end{align}

Since the solution of the reduced problem is $\lambda$ independent, for comparison of the actual problem and the reduced one, we state the 1-D reduced problem as a 2-D one in the same domain $\Omega_f$.

\noindent \textbf{Reduced model}: The flow equation is given by
\begin{equation}
-\nabla_g \cdot f_\beta(\|\nabla_g \bar W\|_g)\nabla_g \bar W  
=\frac{Q}{\left|\Omega\right|} + \frac{1}{2h} \left(q^+(u)+q^-(u) \right)  \mbox{ in } \Omega_f \,, \label{new2Ln}
\end{equation}
with boundary conditions
\begin{align}
\left<f_\beta (\|\nabla_g \bar W\|_g)\nabla_g \bar W , \mathbf{n}_f^\pm\right>_g &= 0,& &\mbox{ on } \Gamma_f^+ \cup \Gamma_f^- \,,\label{bc1aLn}\\
\bar W  &=0 ,& &\mbox{ on } \Gamma_w,\label{bc2aLn}\\
\left<f_\beta (\|\nabla_g \bar W\|_g)\nabla_g \bar W , \mathbf{n}\right>_g &=0,& &\mbox{ on } \Gamma_{f_{out}}\label{bc3aLn} \,.
\end{align}
\end{enumerate}

\begin{theorem}\label{Tdifference}
Let $W$ and $\bar W$ be the solutions of B.V.P.s \eqref{prob_1} and \eqref{prob_2} respectively. Then,
\begin{align}
\left\|\nabla_{g_u} (W-\bar W)\right\|_{L^\frac{3}{2}(\Omega_f)}^{2} + 
 \left\|W_{\lambda}\right\|_{L^\frac{3}{2}(\Omega_f)}^{2}
\leq C \left(\|q^+ \|^2_{L^3(\Omega_f)} +\|q^-\|^2_{L^3(\Omega_f)}\right)\,,
\end{align} for some constant $C$.
 \end{theorem}

 \begin{proof}
Subtracting Eq.~\eqref{new2Ln} from Eq.~\eqref{new1Ln}, multiplying by $z=z(u,\lambda)=W- \bar W$ and integrating over the volume of the fracture we obtain
\begin{align}
&\iint_{\Omega_f}
-\nabla_g \cdot \Big( f_\beta(\|\nabla_g W\|_g)\nabla_g W - f_\beta(\|\nabla_g \bar W\|_g)\nabla_g \bar W \Big)\,z\, d\mu \nonumber \\
&\qquad \qquad \qquad = \iint_{\Omega_f} -\frac{1}{2h} \left(q^++q^- \right)\,z\, d\mu \,.\label{VolInta}
\end{align}
Using Green's formula on Riemannian manifolds (See Ref \cite{do2012differential}) and due to the boundary conditions (\ref{bc3Ln}), (\ref{bc4Ln}), (\ref{bc2aLn}) and (\ref{bc3aLn}), the left hand side of the above equation can be rewritten as
\begin{align}
&\iint_{\Omega_f} \Big\langle\left( f_{\beta}\left( \|\nabla_g W\|_g \right)\nabla_g W - f_{\beta}( \|\nabla_g \bar W\|_g )\nabla_g \bar W \right), \nabla_g z \Big\rangle_g \, d\mu \nonumber \\
&\quad -\int_{\Gamma_f^+} \Big\langle\left( f_{\beta}( \|\nabla_g W\|_g )\nabla_g W - f_{\beta}( \|\nabla_g \bar W\|_g )\nabla_g \bar W \right), \mathbf{n}^+_f \Big\rangle_g \, z\big \rvert_{\Gamma_f^+}\, d\tilde{\mu} \nonumber \\
&\quad -\int_{\Gamma_f^-} \Big\langle\left( f_{\beta}( \|\nabla_g W\|_g )\nabla_g W - f_{\beta}( \|\nabla_g \bar W\|_g )\nabla_g \bar{W} \right), \mathbf{n}^-_f \Big\rangle_g \, z\big \rvert_{\Gamma_f^-}\, d\tilde{\mu} 
%&\quad -\int_{\Gamma_w} \Big\langle\left( f_{\beta}( \|\nabla_g W\|_g )\nabla_g W - f_{\beta}( \|\nabla_g \bar W\|_g )\nabla_g \bar{W} \right), \mathbf{n} \Big\rangle_g  \, z\big \rvert_{\Gamma_w} \, d\tilde{\mu}  \nonumber\\
%&\quad -\int_{\Gamma_{f_{out}}} \Big\langle\left( f_{\beta}( \|\nabla_g W\|_g )\nabla_g W - f_{\beta}( \|\nabla_g \bar W\|_g )\nabla_g \bar{W} \right), \mathbf{n} \Big\rangle_g \, z\big \rvert_{\Gamma_{f_{out}}}\, d\tilde{\mu}  
\,.
\end{align}
In the previous equations, as well as in the rest of the paragraph, $d\mu$ is the Riemannian volume element of the manifold $M$ and $d\tilde{\mu}$ is the area element of the boundary of the manifold $M$. 
%Due to the boundary conditions (\ref{bc3Ln}), (\ref{bc4Ln}), (\ref{bc2aLn}) and (\ref{bc3aLn}), the last two integrals of the above equation are zero.
Eq.~\eqref{VolInta} is equivalent to 
\begin{align}
I_1=I_2+I_3  \label{I1,2,3}\,.
\end{align}
where,
\begin{align}
I_1&=\iint_{\Omega_f} \Big\langle\left( f_{\beta}\left( \|\nabla_g W\|_g \right)\nabla_g W - f_{\beta}( \|\nabla_g \bar W\|_g )\nabla_g \bar{W} \right), \nabla_g z \Big\rangle_g \, d\mu \label{I_1} \,,\\
I_2&=-\iint_{\Omega_f} \frac{1}{2h} \left(q^+(u)+q^-(u) \right)\,z\, d\mu  \,, \nonumber \\
I_3 &=\int_{\Gamma_f^+} \Big\langle\left( f_{\beta}( \|\nabla_g W\|_g )\nabla_g W - f_{\beta}( \|\nabla_g \bar W\|_g )\nabla_g \bar{W} \right), \mathbf{n}^+_f \Big\rangle_g \, z\big \rvert_{\Gamma_f^+}\, d\tilde{\mu} \nonumber \\
&\quad +\int_{\Gamma_f^-} \Big\langle\left( f_{\beta}( \|\nabla_g W\|_g )\nabla_g W - f_{\beta}( \|\nabla_g \bar W\|_g )\nabla_g \bar{W} \right), \mathbf{n}^-_f \Big\rangle_g \, z\big \rvert_{\Gamma_f^-}\, d\tilde{\mu} \,.
\end{align}
Consider $I_1$.
By lemma \ref{boundness_n} and \ref{lidia_n} with $q=3/2$, there exist positive constant $C_0$ such that,
 \begin{align}
I_1 &\geq C_0 
 \left\|\nabla_g z \right\|_{L^\frac{3}{2}}^{2} = 
 C_0  \left( 
 \left\|\nabla_{g_u} z\right\|_{L^\frac{3}{2}}^{2} + 
 \left\|W_{\lambda}\right\|_{L^\frac{3}{2}}^{2}
  \right)\,.
 \label{I_1newn}
\end{align}
Now consider $I_3$.
Due to the boundary conditions \eqref{bc1Ln} and since $\bar W$ is $\lambda$ independent, we have
\begin{comment}
\begin{align}
 f_\beta \left(\|\nabla W _f\| \right)& \nabla W_f \cdot \mathbf{n}_f^+ (W_f- \bar{W}_f ) \Big \rvert_{\frac{h}{2}}- f_\beta \left(\|\nabla\bar W _f\| \right) \nabla \bar W_f  \cdot \mathbf{n}_f^+ (W_f- \bar{W}_f ) \Big \rvert_{\frac{h}{2}}  \nonumber \\
 =& -q^+(x)(W_f- \bar{W}_f ) \Big \rvert_{\frac{h}{2}}\,,
\end{align}
and
\begin{align}
 f_\beta \left(\|\nabla W _f\| \right) & \nabla W_f \cdot \mathbf{n}_f^- (W_f- \bar{W}_f ) \Big \rvert_{-\frac{h}{2}}-  f_\beta \left(\|\nabla\bar W _f\| \right) \nabla \bar W_f  \cdot \mathbf{n_f^-} (W_f- \bar{W}_f ) \Big \rvert_{-\frac{h}{2}}  \nonumber \\
 =& -q^-(x)(W_f- \bar{W}_f ) \Big \rvert_{-\frac{h}{2}}\,.
\end{align}
\end{comment}
\begin{align}
I_3= \int_{\Gamma_f^+} q^+(u) \, W\big \rvert_{\Gamma_f^+}\, d\tilde{\mu} +\int_{\Gamma_f^-}q^-(u) W\big \rvert_{\Gamma_f^-}\, d\tilde{\mu} \,.
\end{align}
It then follows that
\begin{align}
I_2+I_3=\int_{\Gamma_f^+} q^+(u) \int_{-h}^{h}\frac{\left( W\rvert_{\Gamma_f^+} - W\right)}{2h}\,\,d\lambda\,d\tilde \mu +
\int_{\Gamma_f^-} q^-(u) \int_{-h}^{h}\frac{\left( W\rvert_{\Gamma_f^-} - W\right)}{2h}\,\,d\lambda\,d\tilde \mu
 \,. \label{I2+3Ln}
\end{align}
Using H\"older and Cauchy inequalities we obtain
\begin{align}
| I_2+I_3| &\le
\left( \int_{\Gamma_f^+} \int_{-h}^{h} \left|q^+\right|^3   d\lambda  d\tilde\mu \right)^\frac{1}{3}
\left(\int_{\Gamma_f^+} \int_{-h}^{h} \left| \frac{W\rvert_{\Gamma_f^+} - W}{2h} \right|^\frac{3}{2} d\lambda  d\tilde\mu\right)^\frac{2}{3}\nonumber \\
&+
\left( \int_{\Gamma_f^-} \int_{-h}^{h} \left|q^-\right|^3   d\lambda  d\tilde\mu \right)^\frac{1}{3}
\left(\int_{\Gamma_f^-} \int_{-h}^{h} \left| \frac{W\rvert_{\Gamma_f^-} - W}{2h} \right|^\frac{3}{2} d\lambda  d\tilde\mu\right)^\frac{2}{3}\nonumber \\
&\leq 
\frac{1}{4 \varepsilon}
 \Bigg[ \left( \int_{\Gamma_f^+} \int_{-h}^{h}  \left|q^+\right|^3  d\lambda  d\tilde\mu\right)^\frac{2}{3}+
 \left(\int_{\Gamma_f^-} \int_{-h}^{h}  \left|q^-\right|^3    d\lambda  d\tilde\mu \right)^\frac{2}{3} \Bigg]\nonumber \\
 &+ \varepsilon\Bigg[ \left(\int_{\Gamma_f^+} \int_{-h}^{h}   \left| \frac{W\rvert_{\Gamma_f^+} - W}{2h}  \right|^\frac{3}{2} d\lambda  d\tilde\mu\right)^\frac{4}{3}+
\left(\int_{\Gamma_f^-} \int_{-h}^{h}  \left| \frac{W\rvert_{\Gamma_f^-} - W}{2h}  \right|^\frac{3}{2} d\lambda  d\tilde\mu \right)^\frac{4}{3}\Bigg] \,.
\end{align}
Using Poincar\'e inequality,
we have 
\begin{align}
\int_{\Gamma_f^+} \int_{-h}^{h}   \left| \frac{W\rvert_{\Gamma_f^+} - W}{2h}  \right|^\frac{3}{2} d\lambda  d\tilde\mu
\leq 
\frac{2}{3}\int_{\Gamma_f^+} \int_{-h}^{h}   \left| W_{\lambda} \right|^\frac{3}{2}  d\lambda  d\tilde\mu
\end{align}
and 
\begin{align}
\int_{\Gamma_f^-} \int_{-h}^{h}   \left| \frac{W\rvert_{\Gamma_f^-} - W}{2h}  \right|^\frac{3}{2} d\lambda  d\tilde\mu
\leq 
\frac{2}{3}\int_{\Gamma_f^-} \int_{-h}^{h}   \left| W_{\lambda} \right|^\frac{3}{2}  d\lambda  d\tilde\mu\,.
\end{align}
Therefore,
\begin{align}
 | I_2+I_3| &\le \frac{1}{4 \varepsilon}
 \Bigg[  \left( \iint_{\Omega_f}  \left|q^+\right|^3  d\mu\right)^\frac{2}{3}+
 \left(\iint_{\Omega_f}  \left|q^-\right|^3 d\mu \right)^\frac{2}{3} \Bigg]\nonumber \\
&\qquad +2\varepsilon \left(\frac{2}{3}\right)^\frac{4}{3} \left(\iint_{\Omega_f}   \left| W_{\lambda} \right|^\frac{3}{2}  d\mu\right)^\frac{4}{3}\,. \label{modI2+3Ln}
\end{align}
Combining Eqs.~\eqref{I_1newn}, \eqref{I2+3Ln} and \eqref{modI2+3Ln}, choosing
$\varepsilon=\frac{C_0}{4}\left(\frac{3}{2}\right)^\frac{4}{3}$ and setting $C_1=\frac{1}{C_0} \left(\frac{2}{3}\right)^\frac{4}{3} $
yields
\begin{align}
& C_0 \left\|\nabla_{g_u} z\right\|_{L^\frac{3}{2}(\Omega_f)}^{2} + 
 \frac{C_0}{2}\left\|W_{\lambda}\right\|_{L^\frac{3}{2}(\Omega_f)}^{2}
 \le
C_1 \left( \|q^+ \|^2_{L^3(\Omega_f)} +\|q^-\|^2_{L^3(\Omega_f)} \right).
\end{align}
Therefore, we have
\begin{align}
\left\|\nabla_{g_u} z\right\|_{L^\frac{3}{2}(\Omega_f)}^{2} + 
 \left\|W_{\lambda}\right\|_{L^\frac{3}{2}(\Omega_f)}^{2}
\leq C \left(\|q^+ \|^2_{L^3(\Omega_f)} +\|q^-\|^2_{L^3(\Omega_f)}\right)\,,
\end{align}
with $C=2 C_1/C_0$.
\end{proof}

\begin{remark}
From the theorem above, it follows that for a given fracture with thickness $h$, 
the difference between the solutions of the two problems can be controlled by the boundary data.
\end{remark} 

However, it should be noted that in the reservoir-fracture system as $h$ goes to zero, the fracture vanishes, and the oil 
flows mostly towards the well. Then as $h$ becomes smaller, $q^+$ and $q^-$ gets smaller as well, and therefore the individual velocities remain bounded.

Next, we investigate the coupled fractured-porous media domain. We show a much stronger result, 
under the condition that as $h\rightarrow 0$ the fluxes on the fracture boundary vanish with the same speed.

\subsection{Estimates for the difference between the solutions in coupled fracture-porous media domain with linear isotropic flows.}
Next, we provide estimates for the difference between the solutions in coupled domain.
We consider half of a symmetric idealized fracture-reservoir domain depicted below and consider the flow to be linear isotropic inside the fracture. Let $\Omega_p$ and $\Omega_f$ be the porous media region and the fracture, respectively, with $\Omega=\Omega_p \cup \Omega_f$. Let $\Gamma_f$ be the top boundary of the fracture
described by $\mathbf{R}(u,v,h)$, $\Gamma_w$ be the well boundary, $\Gamma_{out}$ be the outer boundary of $\Omega$,  $\Gamma_{f_{out}}$ be the right extremum of the fracture. Let $\mathbf{n}$ and $\mathbf{n}_f$ be the outward unit normal on $\Gamma_{out}$ and $\Gamma_f$ respectively. We build the domain $\Omega_f$ such that $\mathbf{R}(u,v,\lambda)= \mathbf{r}(u,v) + \lambda \, \mathbf{n}_f$, with
$0\le\lambda \le h$. Namely, we get the profile $\mathbf{r}(u,v)$ imposing $ \mathbf{r}(u,v) + h\, \mathbf{n}_f  = \mathbf{R}(u,v,h) $, with $\mathbf{R}(u,v,h)$ given.

Let $k_p$ and $k_f$ be the permeability of the porous media and the fracture respectively. Let $W_i$ be the flow pressure in the original problem and $\bar W_i$ be the flow pressure in the reduced problem with $i \in \{p,f\}$. Here $p$ denotes the porous media and $f$ denotes the fracture.
Let $\bar q(u)$ be the flux coming into the fracture from the reservoir.
 
\begin{figure}[!htb]
\begin{center}
\includegraphics[scale=0.4]{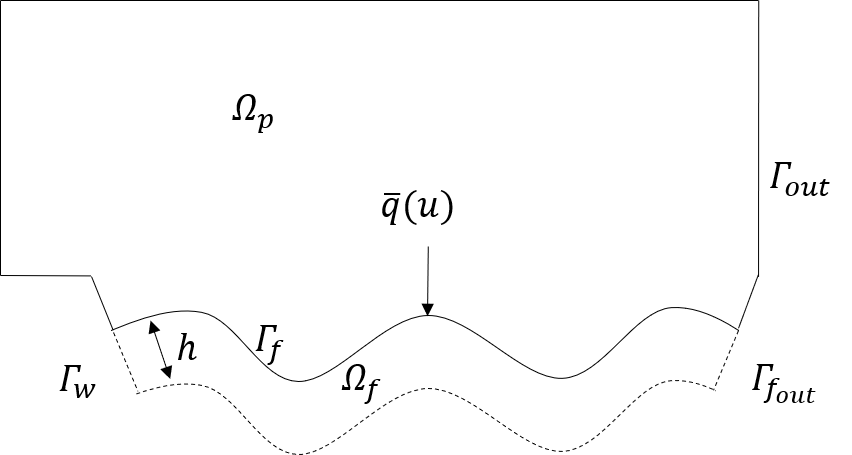}
\caption{Domain of the reduced model.}
 \label{reduced_g}	
\end{center} 
\end{figure}
 We investigate the difference between the solutions of the two problems defined below with $\nabla_g = <\frac{1}{1+f_u^2}\partial_u,\partial_{\lambda}>$ and $\nabla_{g_u} = \frac{1}{1+f_u^2}\partial_u.$
 \begin{enumerate}[(I)]
\item \label{p_1g}
The flow equations for the original problem are given by
\begin{comment}
\begin{figure}[!htb]
\begin{center}
\includegraphics[scale=0.4]{Original_g_new.png}
\caption{Domain of the original model.}
 \label{original_g}	
\end{center}
\end{figure}
\vspace{-5mm}
\end{comment}
\begin{align}
-\nabla \cdot k_p \nabla W_p&= \frac{Q}{\left|\Omega\right|}   &&  \mbox{ in } \Omega_p\,,\label{new11g} \\
-\nabla_g \cdot f_\beta(\|\nabla_g W_f\|_g)\nabla_g W_f &= \frac{Q}{|\Omega|}  && \mbox{ in } \Omega_f    \,,\label{new1g}
\end{align}
with boundary conditions
\begin{align}
W_p&=W_f, &  &\mbox{ on } \Gamma_f  \,,\label{bcaaag}\\
\left< k_p\nabla W_p, \mathbf{n}_f \right> &=
\left< f_\beta(\|\nabla_g W_f\|_g)\nabla_g W_f, \mathbf{n}_f \right>_g, &  &\mbox{ on } \Gamma_f  \,,\label{bcbbbg}\\
W_p&=W_f=0,&  &\mbox{ on } \Gamma_w  \,,\label{bc2g} \\
\left<k_p \nabla W_p,  \mathbf{n}\right>  &=0,
&&\mbox{ on }  \Gamma_{out}  \,,\label{bc3g}\\
\left< f_\beta(\|\nabla_g W_f\|_g) \nabla_g W_f,  \mathbf{n}\right>_g  &=0,
&&\mbox{ on } \Gamma_{f_{out}}  \,.\label{bc4g}
\end{align}

\item \label{p_2g}
The flow equations for the Reduced Model is given by
\begin{comment}
\begin{figure}[!htb]
\begin{center}
\includegraphics[scale=0.4]{reduced_g_new.png}
\caption{Domain of the reduced model.}
 \label{reduced_g}	
\end{center}
\end{figure}
\end{comment}
\begin{align}
-\nabla \cdot k_p \nabla \bar W_p&= \frac{Q}{\left|\Omega\right|}   &&  \mbox{ in } \Omega_p\,, \label{new22g} \\
-\nabla_g \cdot  f_\beta(\|\nabla_g \bar{W}_f\|_g) \nabla_g \bar W_f & =\frac{Q}{\left|\Omega\right|} + \frac{\bar q}{h}   && \mbox{ in } \Omega_f    \,, \label{new2g}
\end{align}
with boundary conditions \vspace{-5mm}
\begin{align}
\bar W_p&=\bar W_f, &  &\mbox{ on } \Gamma_f  \,,\label{bcaaa2g}\\
\bar W_p&=\bar W_f=0,&  &\mbox{ on } \Gamma_w  \,,\label{bc2ag} \\
\left<k_p \nabla \bar W_p , \mathbf{n}_f \right> &=\bar q,
&&\mbox{ on } \Gamma_{f}  \,,\label{bc1ag}\\
\left<k_p \nabla \bar W_p,  \mathbf{n}\right>  &=0,
&&\mbox{ on }  \Gamma_{out}  \,,\label{bc3ag}\\
\left< f_\beta(\|\nabla_g \bar W_f\|_g) \nabla_g \bar W_f,  \mathbf{n}\right>_g  &=0,
&&\mbox{ on } \Gamma_{f_{out}}  \,.\label{bc4ag}
\end{align}
\end{enumerate}

For simplicity, we drop the subscripts in $W_p$, $W_f$, $\bar W_p$ and $\bar W_f$ and use notations $W$ and $\bar W$ for the solutions in the original problem \eqref{p_1g} and reduced problem \eqref{p_2g} respectively. Each of the solution corresponds to their domain of integration.

We assume that, for a manifold with appropriate conditions on Riemannian metric, the following conjecture to be true.

\begin{conjecture}\label{conj}
For $\bar q(u)$ given by Eq.~\eqref{bc1ag}, $\left|\dfrac{\bar q(u)}{h}\right|\leq C_0$ for some constant $C_0$ and for all $(u,v,\lambda)\in \Gamma_f$. 
\end{conjecture}

Henceforth $d\mu$ and $ d \Omega$ are the elementary volumes of the manifold and the Euclidean domain respectively. $d\tilde \mu$ and $dS$ are the elementary areas of the manifold and the Euclidean domain respectively. By construction, the fracture domain $\Omega_f$ has the following property: the element of volume and the element of area of the fracture in Euclidean domain and on manifold are exactly the same, namely, $d\Omega=d\mu$ and $dS=d\tilde\mu$.

\begin{theorem} \label{TdifferenceCoupled} 
Let $W$ and $\bar W$ be the solutions of B.V.P.s \eqref{p_1g} and \eqref{p_1g} respectively. Then, under the assumption in conjecture \eqref{conj} the following estimate holds for some constant $C$ that depends on $C_0$.
\begin{align}
\|\nabla \left(W-\bar W\right)\|_{L^2(\Omega_p)}^2  \leq C h^\frac{8}{3} \,.
\end{align}
\end{theorem}
%\begin{lemma}
%\begin{align}
%\iint_{\Omega_p} k_p  \left(\nabla z\right)^2 d \Omega
%+\iint_{\Omega_f}k_f  \left\|\nabla_g\, z\right\|_g^2 d \mu
%=\iint_{\Omega_f} \frac{\bar q}{h} \left(W\big \rvert_{\Gamma_f}-W\right) d \mu\,.
%\end{align}
%\end{lemma}
\begin{proof}
Let $z=z(u,v,\lambda) = W- \bar W$. Subtracting Eq.~\eqref{new22g} and ~\eqref{new2g} from Eq.~\eqref{new11g} and ~\eqref{new1g}, respectively, we obtain
\begin{align}
-\nabla \cdot k_p \nabla z&= 0   &&  \mbox{ in } \Omega_p\,,\label{P1g} 
\end{align} and
\begin{align}
-\nabla_g \cdot \left(f_\beta(\|\nabla_g W\|_g)\nabla_g W-
f_\beta(\|\nabla_g \bar W\|_g)\nabla_g \bar W\right) &= -\frac{\bar q}{h} && \mbox{ in } \Omega_f    \,.\label{P2g}
\end{align}
Then, after adding Eqs.~\eqref{P1g} and \eqref{P2g}, multiplying by $z$ and integrating over the volume of the domain we obtain,

\begin{align}
\iint_{\Omega_p}-\nabla \cdot k_p  \nabla z \,z\, d \Omega
+&\iint_{\Omega_f}-\nabla_g \cdot \left(f_\beta(\|\nabla_g W\|_g)\nabla_g W-
f_\beta(\|\nabla_g \bar W\|_g)\nabla_g \bar W\right)\,z d \mu\nonumber \\
&= -\iint_{\Omega_f} \frac{\bar q}{h} z\, d \mu\,. \label{eq000}
\end{align}
Using Green's formula on Riemannian manifolds (See Ref \cite{do2012differential}) and boundary conditions \eqref{bc2g}, \eqref{bc3g}, \eqref{bc4g}, \eqref{bc2ag}, \eqref{bc3ag} and \eqref{bc4ag}, the first and the second integrals in the left hand side of the above equation can be rewritten as
\begin{align}
\iint_{\Omega_p}-\nabla \cdot k_p  \nabla z \,z\, d \Omega=
\iint_{\Omega_p} k_p  \left(\nabla z\right)^2 d \Omega
-\int_{\Gamma_f} \left< k_p \nabla z, (-\mathbf{n}_f) \right> z\big \rvert_{\Gamma_f}d S\,,\label{i1}
\end{align}
and 
\begin{align}
\iint_{\Omega_f}-\nabla_g &\cdot \Big(f_\beta(\|\nabla_g W\|_g\nabla_g W -
f_\beta(\|\nabla_g \bar W\|_g)\nabla_g \bar W\Big)\,z d \mu \nonumber \\
=&\iint_{\Omega_f}\left<\Big(f_\beta(\|\nabla_g W\|_g\nabla_g W -
f_\beta(\|\nabla_g \bar W\|_g)\nabla_g \bar W\Big) , \nabla_g z\right>_g d \mu \nonumber \\
-&\int_{\Gamma_f} \left< \Big(f_\beta(\|\nabla_g W\|_g\nabla_g W -
f_\beta(\|\nabla_g \bar W\|_g)\nabla_g \bar W\Big) , \mathbf{n}_f \right>_g z\big \rvert_{\Gamma_f}d \tilde \mu\,,\label{i2}
\end{align} respectively.\\
Combining Eqs.~\eqref{i1}, \eqref{i2}, boundary condition \eqref{bcbbbg} and since $\Big< f_\beta(\|\nabla_g \bar W\|_g)\nabla_g \bar W, \mathbf{n}_f \Big>_g=0$, Eq.~\eqref{eq000} can be rewritten as
\begin{align}
\iint_{\Omega_p} k_p  \left(\nabla z\right)^2 d \Omega
+\iint_{\Omega_f}& \left<\Big(f_\beta(\|\nabla_g W\|_g\nabla_g W -
f_\beta(\|\nabla_g \bar W\|_g)\nabla_g \bar W\Big) , \nabla_g z\right>_g d \mu \nonumber \\
&=\int_{\Gamma_f} \left<k_p \nabla \bar W, \mathbf{n}_f \right> z\big \rvert_{\Gamma_f}d S
-\iint_{\Omega_f} \frac{\bar q}{h} z \, d \mu\,.
\end{align}
Due to boundary condition \eqref{bc1ag} and since $dS=d\tilde \mu$, we have
\begin{align}
\iint_{\Omega_p} k_p  \left(\nabla z\right)^2 d \Omega
+\iint_{\Omega_f} & \left<\Big(f_\beta(\|\nabla_g W\|_g\nabla_g W -
f_\beta(\|\nabla_g \bar W\|_g)\nabla_g \bar W\Big) , \nabla_g z\right>_g d \mu \nonumber \\
&= \int_{\Gamma_f} \bar q\, z\big \rvert_{\Gamma_f}d \tilde \mu  -\iint_{\Omega_f} \frac{\bar q}{h}z\,d \mu \nonumber \\
&= \iint_{\Omega_f} \frac{\bar q}{h} \left(z\big \rvert_{\Gamma_f}-z\right)d \mu\nonumber \\
&=\iint_{\Omega_f} \frac{\bar q}{h} \left(W\big \rvert_{\Gamma_f}-W\right)d \mu\,.\label{ori_diffg}
\end{align}
Denote \begin{align}
I_1&=  k_p \iint_{\Omega_p} \left(\nabla z\right)^2 d \Omega = k_p\|\nabla z\|_{L^2(\Omega_p)}^2\,, \nonumber \\
I_2&= \iint_{\Omega_f} \left<\Big(f_\beta(\|\nabla_g W\|_g\nabla_g W -
f_\beta(\|\nabla_g \bar W\|_g)\nabla_g \bar W\Big) , \nabla_g z\right>_g d \mu\,, \nonumber \\
I_3 &= \iint_{\Omega_f}\frac{\bar q}{h} \left(W\big \rvert_{\Gamma_f}-W\right)d \mu\,.
\end{align}
Then,
\begin{align}
I_1+I_2=I_3\,. \label{I123g}
\end{align}
Consider $I_2$.
By Lemma \ref{boundness_n} and \ref{lidia_n} with $q=\frac{3}{2}$, there exist positive constant $C_1$ such that,
 \begin{align}
I_2 &\geq C_1 
 \left\|\nabla_g z \right\|_{L^\frac{3}{2}(\Omega_f)}^{2} = 
 C_1  \left( 
 \left\|\nabla_{g_u} z\right\|_{L^\frac{3}{2}(\Omega_f)}^{2} + 
 \left\|W_{\lambda}\right\|_{L^\frac{3}{2}(\Omega_f)}^{2}
  \right)\,.
\label{I2g}
\end{align}
\begin{comment}
\begin{align}
I_2&\geq C_1  \iint_{\Omega_f} \left\|\nabla_g z\right\|_g^2d \mu\\
&= k_f \left( \iint_{\Omega_f} \left\|\nabla_{g_u} z\right\|_{g_u}^2d \mu +\iint_{\Omega_f} z_{\lambda}^2d\mu\right)\nonumber \\
&=k_f \left( \iint_{\Omega_f} \left\|\nabla_{g_u} z\right\|_{g_u}^2d \mu +\iint_{\Omega_f} W_{\lambda}^2d\mu\right)\,, \label{I2g}
\end{align}
for some constant $C_1$.
\end{comment}
Also, from conjecture \eqref{conj}, we have 
\begin{align}
\left|I_3\right|\leq & C_0 \iint_{\Omega_f}\left|W\big \rvert_{\Gamma_f}-W\right|d \mu \nonumber \\
& =C_0 \iint_{\Omega_f}\left|W\big \rvert_{\Gamma_f}-W\right|d \Omega
\end{align}
Using H\"older and Cauchy inequalities we obtain
\begin{align}
\left|I_3\right| &\leq C_0 \left(\iint_{\Omega_f}\left|W\big \rvert_{\Gamma_f}-W\right|^\frac{3}{2} d \Omega\right)^\frac{2}{3}\left(\iint_{\Omega_f}d \Omega\right)^\frac{1}{3} \nonumber \\
 &\leq \frac{C_0}{4\epsilon}\left(\iint_{\Omega_f}\, d \Omega \right)^\frac{2}{3}+ C_0\epsilon \left(\iint_{\Omega_f}\left|W\big \rvert_{\Gamma_f}-W\right|^\frac{3}{2}d \Omega \right)^\frac{4}{3}\nonumber \\
 &\leq \frac{C_0}{4\epsilon}\left|\Omega_f\right|^\frac{2}{3} + C_0\epsilon \left(\int_{\Gamma_f} \int_0^h\left|W\big \rvert_{\Gamma_f}-W\right|^\frac{3}{2} d \lambda\, d \tilde \mu \right)^\frac{4}{3} \,,
\end{align} where $\left|\Omega_f\right|$ is the volume of the fracture.
Then, using Poincar\'e inequality, we have
\begin{align}
\int_{\Gamma_f} \int_0^h\left|W\big \rvert_{\Gamma_f}-W\right|^\frac{3}{2} d \lambda\, d \tilde \mu
\leq 
\frac{2}{3}h^\frac{3}{2}\int_{\Gamma_f} \int_{0}^{h}   \left| W_{\lambda} \right|^\frac{3}{2}  d\lambda  d\tilde\mu\,.
\end{align}
Therefore,
\begin{align}
\left|I_3\right|&\leq \frac{C_0}{4\epsilon}\left|\Omega_f\right|^\frac{2}{3}+ C_0\varepsilon \left(\frac{2}{3}\right)^\frac{4}{3}h^2 \left(\iint_{\Omega_f} \left| W_{\lambda} \right|^\frac{3}{2}  d\mu\right)^\frac{4}{3} \nonumber \\
&= \frac{C_0}{4\epsilon}\left|\Omega_f\right|^\frac{2}{3}+C_0\varepsilon \left(\frac{2}{3}\right)^\frac{4}{3}h^2\,\|W_\lambda\|_{L^\frac{3}{2}(\Omega_f)}^2 \,.
\label{I3g}
\end{align}
Combining Eqs.~\eqref{I123g}, \eqref{I2g}, \eqref{I3g} we obtain
\begin{align}
 & k_p \|\nabla z\|_{L^2(\Omega_p)}^2 +
 C_1 \left\|\nabla_{g_u} z\right\|_{L^\frac{3}{2}(\Omega_f)}^{2} + 
 C_1 \left\|W_{\lambda}\right\|_{L^\frac{3}{2}(\Omega_f)}^{2} \nonumber \\ 
 & \qquad \leq \frac{C_0}{4\epsilon}\left|\Omega_f\right|^{\frac{2}{3}}+ C_0\varepsilon \left(\frac{2}{3}\right)^\frac{4}{3}h^2\, \|W_\lambda\|_{L^\frac{3}{2}(\Omega_f)}^2\,.
\end{align}
Choose $\epsilon=\frac{1}{h^2}\frac{C_1}{C_0}\left(\frac{3}{2}\right)^{\frac{4}{3}}$. Then
\begin{align}
k_p \|\nabla z\|_{L^2(\Omega_p)}^2  \leq C h^2 \left|\Omega_f\right|^\frac{2}{3}\,,
\end{align}
with $\tilde C=\dfrac{C_0^2}{4C_1}\left(\dfrac{2}{3}\right)^\frac{4}{3}$. Therefore,
\begin{align}
\|\nabla z\|_{L^2(\Omega_p)}^2  \leq C h^\frac{8}{3} \,.
\end{align}
where $\left|\Gamma_f\right|$ is the surface area of the boundary $\Gamma_f$ and $C=\dfrac{\tilde C}{k_p}\left|\Gamma_f\right|^\frac{2}{3}.$
\end{proof}
\begin{remark}
From the theorem above, it can be observed that the estimate for the difference between the solutions of the original problem
and the reduced problem goes to zero as $h$ going to zero, representing a much stronger estimate.
\end{remark}
%%%%%%%%%%%%%%%%%%%%%%%%%%%%%

%%%%%%%%%%%%%%%%%%%%%%%%%%%%%%%%%%%%%%%%%%%%

\section{Numerical analysis and simulations}\label{NA}
In this section, we present numerical results to demonstrate the validity of our approach. Namely, we show that the solutions of the Reduced Model I \eqref{FINAL_THM} and the Reduced Model II \eqref{red_Thm} are close to the solution of the original problem \eqref{actual} for  different fracture-reservoir geometries. 
%In particular, we first compare the pressure distributions in the domain with fracture only (subsection \ref{1}), and then, we make comparisons in the fully coupled fractured reservoir domain (subsection \ref{2}). In all the examples, the difference between the solutions (for pressure and diffusive capacity) of the original model and the reduced models are very small, verifying that the reduced models accurately approximate the solutions of the original model.

All the simulations have been performed using COMSOL Multiphysics software \cite{bv1998comsol}. The grid size has been refined until changes in the pressure distribution between the previous and the next steps are negligible. 
The length and the thickness of the fracture is selected in relative units, and are dimensionless. Hydrodynamic parameters such as permeability and Forchheimer coefficient are numerically chosen without bonding to actual data of the porous media properties. In all simulations, the following parameters have been fixed: permeability in the porous media $k_p=0.01$, permeability inside the fracture $k_f=1$ and production rate $Q=1$.

\subsection{Pressure distribution of the flow inside the fracture}\label{1}
In this section, we compare the pressure distributions of the flow obtained from the original model, the Reduced Model I and the Reduced Model II, inside the domain of fracture only, for different fracture geometries.
For both Examples \ref{Ex1} and \ref{Ex4} we perform the following numerical simulations:
First, we obtain the solution of the original model \eqref{actual} inside the fracture, imposing zero Dirichlet boundary condition on the well, given flux boundary conditions on the top and bottom of the fracture, and zero Neumann  boundary condition on the right end (namely, system \eqref{well}-\eqref{noFluxBdry} with $q^\pm=10$). 
Then, we solve the Reduced model I \eqref{FINAL} and Reduced Model II \eqref{FINAL2} on the barycentric line of the fracture cross section, with zero Dirichlet boundary condition on the well and zero Neumann boundary condition on the right end. 

%Reduced Model I is more comprehensive and the coefficients implicitly depend on $\lambda$.
%We show that the solution of the Reduced Model I is close to the solution of the original model, for values of the fracture thickness $h\leq h_0$, for $h_0$ \textit{relatively big}.  
%Reduced Model II is a simplified version of the Reduced Model I in which the coefficients of the equation do not depend on $\lambda$. Consequently, the corresponding solution does not depend on $\lambda$. Therefore, for small values of the fracture thickness, the solution obtained from the Reduced Model II will remain close to the original one, but as the fracture thickness becomes large the two solutions will deviate from each other. We numerically investigate this issue in detail.   
%In Examples \ref{Ex1} and \ref{Ex2}, we compare the pressure distributions of the flow in thin fractures using the original model and the Reduced Model II. In Examples \ref{Ex3}, \ref{Ex4} and \ref{Ex5}, we compare the pressure distributions of the flow inside the fractures using the original model, the Reduced Model I and the Reduced Model II, for large thicknesses and changing thicknesses.

\begin{example} \label{Ex1} {\rm  
We consider the fracture geometry with barycentric surface given by  $$\mathbf{r}(u,v)= \langle u,v,\sqrt{1-u^2} \rangle, \mbox{ with } (u,v)\in \left[-\frac{\sqrt{3}}{2}, \frac{\sqrt{3}}{2}\right]\times(-\infty,\infty),$$ and constant thickness $2h(u)$.
Since the solution of the problem does not depend on $v$, we solve our equations only on the cross section given in Figure \ref{half_circle}.

%%%%%%%%%%%%%%%%%%%%%%%%%%%%
\begin{comment}}
\begin{figure}[!htb]
\begin{tabular}{c}
\includegraphics[scale=0.3]{half_circle_col.png}
\end{tabular}
\end{figure}
\end{comment}
%%%%%%%%%%%%%%%%%
\begin{figure}[h!]
\begin{minipage}{0.48\textwidth}\vspace{-4.5mm}
\includegraphics[scale=0.25]{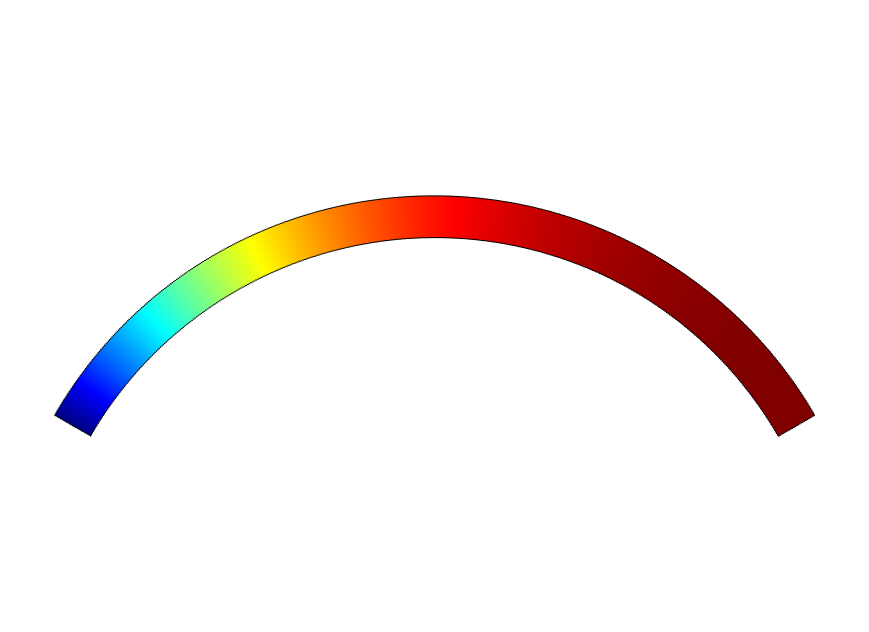}
\end{minipage}
\hspace{1mm}
\begin{minipage}{0.48\textwidth}
\vspace{-10mm}
\includegraphics[scale=0.3]{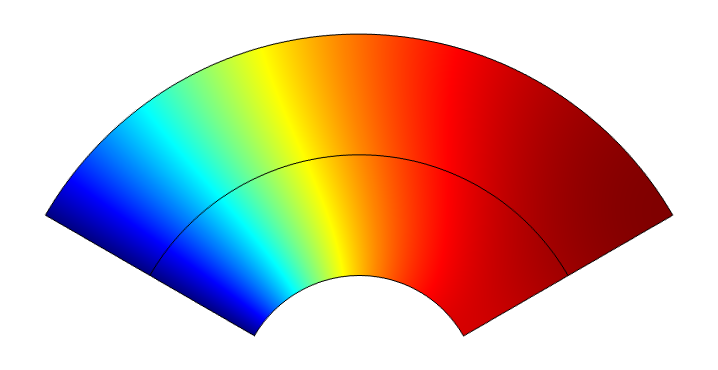} \vspace{-12mm}
\end{minipage}
\caption{Fracture geometry 1: pressure distribution inside a thin fracture-$2h(u)=0.025$ (left) and a thick fracture-$2h(u)=1$ (right), obtained using the original model.}
 \label{half_circle}	
\end{figure}
%%%%%%%%%%%%%%%%%%%%%%%%%%%%

In Figure \ref{F1}, we compare the pressure distributions obtained from the original model (on the barycentric line of the fracture cross section) and the Reduced Models for $\beta=0.1$, $2h(u)=0.025$ (left) and $2h(u)=1$ (right). It is evident that the solutions are almost identical to each other when the thickness of the fracture is relatively small. When the thickness is large, it can be observed that the solution obtained from the Reduced Model I is very close to the solution of the original model, but the solution of the Reduced Model II deviates from the original one. This can be explained by the presence versus the absence of the parameter $\lambda$ in the Reduced Model I and II, respectively, as explained in Remark \ref{remark_lambda}. 
%Therefore, for small values of the fracture thickness, the solution obtained from the Reduced Model II will remain close to the original one, 
%%but as the fracture thickness becomes large the two solutions will deviate from each other. We numerically investigate this issue in detail. 
%while the solution of the Reduced Model I will remain close to the solution of the original model, for values of the fracture thickness that are \textit{relatively big}.  
%%%%%%%%%%%%%%%%%%%%%%%%%%%
\begin{comment}
\begin{figure}[!htb]
\begin{center}	\includegraphics[height=5.5cm,width=8.5cm]{hcgraphNEW.png}
	\caption{Fracture geometry 1: comparison between the pressure distributions obtained from the original model and the Reduced Model II for $\beta=0.1$.}	\label{F1}
\end{center}
\end{figure}
\end{comment}
%%%%%%%%%%%%%%%%%%%%%%%
\begin{figure}[h!]
\begin{minipage}{0.48\textwidth}
\includegraphics[width=1\linewidth, height=5.5cm]{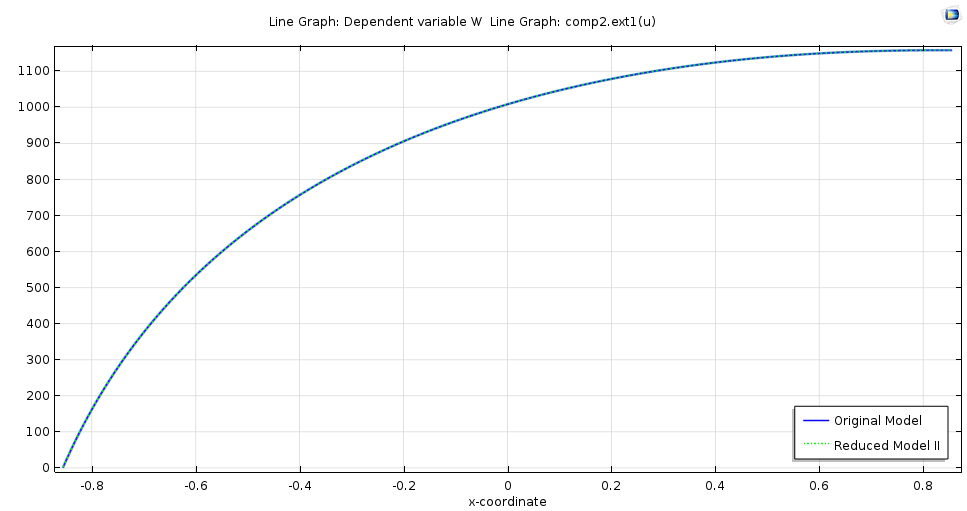}
\end{minipage}
\hspace{1mm}
\begin{minipage}{0.48\textwidth}
\includegraphics[width=1\linewidth, height=5.5cm]{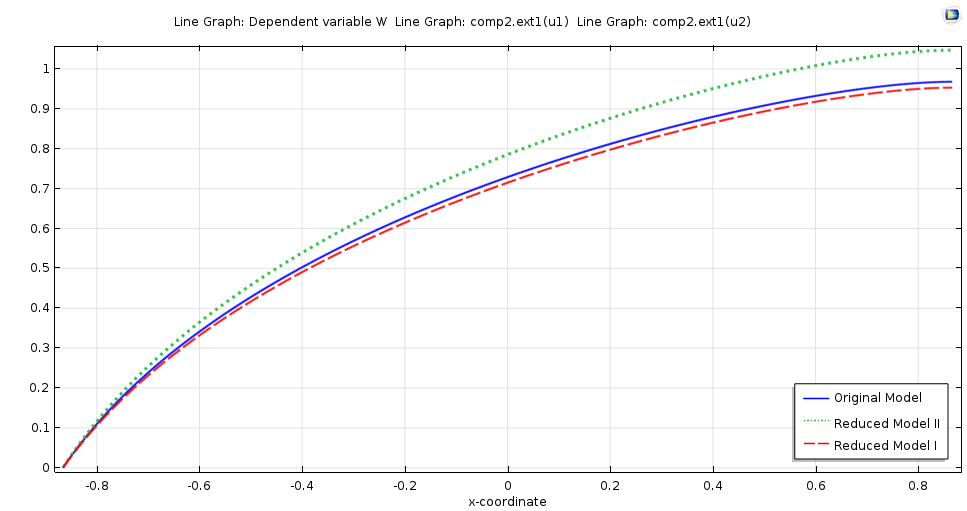}
\end{minipage}
\caption{Fracture geometry 1: comparison between the pressure distributions obtained from the original model, the Reduced  I and Reduced Model II for $\beta=0.1$, $2h(u)=0.025$ (left) and $2h(u)=1$ (right) respectively.} \label{F1}
\end{figure}\vspace{1mm}}
%%%%%%%%%%%%%%%%%%%%%%%%%%%%5
\end{example}

\begin{example} \label{Ex4} {\rm 
Next, we consider the fracture geometry with barycentric surface given by 
$$\mathbf{r}(u,v)=\left \langle u,v,2\sin(u)\right \rangle, \mbox{ with } (u,v)\in \left[0,2\pi\right]\times(-\infty,\infty),$$ and variable thickness $2h(u)=0.2(2+0.5\sin(7u))$.
Again, since the solution of the problem does not depend on $v$, we solve our equations only on the cross section given in Figure \ref{idk}.
\begin{figure}[h]
	\begin{center}
	\includegraphics[scale=0.4]{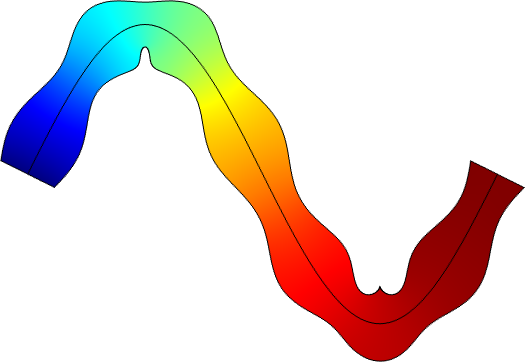}
\caption{Fracture geometry 4: pressure distribution inside a fracture with changing thickness, obtained using the original model.}
\label{idk}	
\end{center}
\end{figure}
%%%%%%%%%%%%%%%%%%%%%%%%%

%We perform numerical simulations as described above. Namely, first, we obtain the solution of the original model \eqref{actual} inside the fracture, imposing zero Dirichlet boundary condition on the well, given flux boundary conditions on the top and bottom of the fracture, and zero Neumann  boundary condition on the right end (namely, system \eqref{well}-\eqref{noFluxBdry} with $q^\pm=10$). 
%Then, we solve the Reduced Model I \eqref{FINAL} and the Reduced Model II \eqref{FINAL2} on the barycentric line of the fracture cross section, with zero Dirichlet boundary condition on the well and zero Neumann boundary condition on the right end. 
\pagebreak
In this example, the thickness of the fracture changes depending on the parameter $u$. In Figure \ref{F4}, we compare the pressure distributions obtained from the original model (on the barycentric line of the fracture cross section), the Reduced Model I and the Reduced Model II, for $\beta=0$ (Darcy) and $\beta=0.1$ (Forchheimer). The solutions of the Reduced Model I remain close to the solutions of the original model. However, the solutions of the Reduced Model II deviate from the solutions of the original model, especially in the Forchheimer case.
\begin{figure}[h!]
\begin{minipage}{0.48\textwidth}
\includegraphics[width=1\linewidth, height=5.5cm]{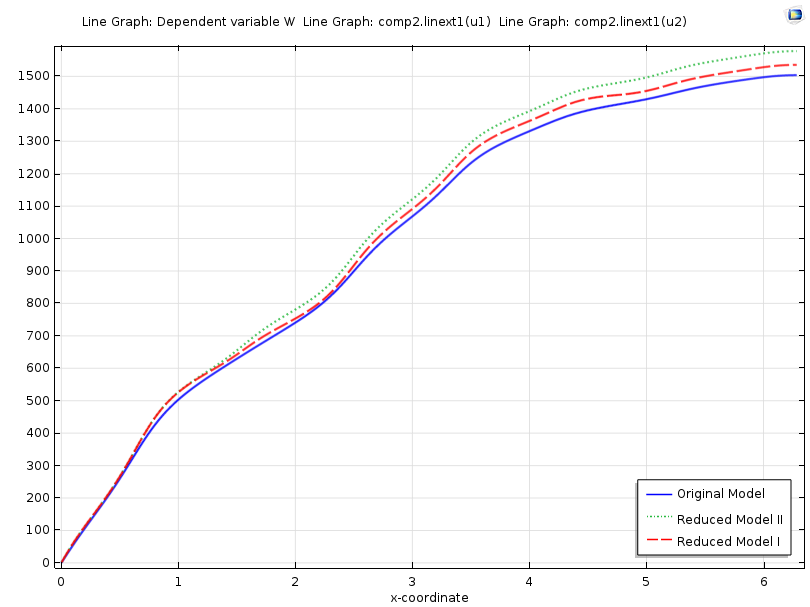}
\end{minipage}
\hspace{1mm}
\begin{minipage}{0.48\textwidth}
\includegraphics[width=1\linewidth, height=5.5cm]{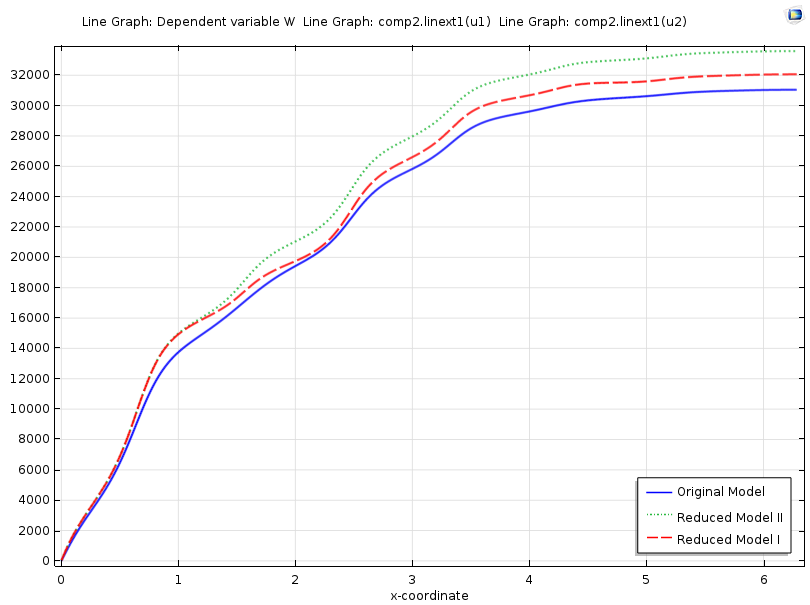}
\end{minipage}
\caption{Fracture geometry 4: comparison between the pressure distributions inside the fracture, obtained using the original model, 
the Reduced Model I and the Reduced Model II, for $\beta=0$ (left) and $\beta =0.1$ (right), respectively.} \label{F4}
\end{figure}}\vspace{1mm}
\end{example}

\subsection{Diffusive capacity in the coupled fractured porous media domain} \label{2}
In this section, we calculate the diffusive capacities in the coupled fracture reservoir domain, using the original model~\eqref{actual1} and the Reduced Model I~\eqref{FINAL}.
%First, we couple the flow in the fracture with the flow in porous media, which allows us to investigate the flow in the fully coupled fractured reservoir domain. The original flow equation inside the fracture \eqref{actual1} is coupled with the flow in the porous media by imposing continuity of the solutions and the fluxes across the fracture boundaries. Here, the fluid flows towards the well $\Gamma_w$ both directly from the porous media and also through the fracture.
%Then, we couple the flow equation in the fracture as in the Reduced Model I, with the flow in porous media. 
%Here, the domain of the flow consists of porous media only. 
%The reduced flow equation inside the fracture \eqref{FINAL} is solved on the boundary of the porous media domain and
%takes into account the nonlinearity of the flow, the
%changing thickness and other geometric properties of the fracture.

%the original problem of the flow from the reservoir towards the well $\Gamma_w$ through the fracture and directly from the porous media to the well $\Gamma_w$. 

The productivity index, which characterizes
the well capacity to \textit{take-in} hydrocarbons from reservoir, depends on the geometry of the fracture and its conductivity, 
and it is evaluated using the diffusive capacity of the well-reservoir-fracture system.
Denote the diffusive capacities of the original model and the Reduced Model I by $PI_{PD}$ and $PI_{R_1}$, respectively. We compare the diffusive capacities obtained from the two models, as the amplitude of the thickness $H$ and Forchheimer coefficient $\beta$ changes, for different geometries of the fracture in the reservoir. Also, we evaluate the relative error as $\mbox{error} = \left|\frac{PI_{PD}-PI_{R_1}}{PI_{PD}}\right|.$
The goal here is to show that, in the fully coupled domain, the diffusive capacity calculated using the Reduced Model I is very close to the diffusive capacity calculated using the original model.

%First we compare the pressure distributions obtained from the original model (Eq.~\eqref{actual}) and the Reduced Model I (Eq.~\eqref{FINAL}) in the coupled domain.

\begin{example}\label{Ex6} {\rm 
In this example, we numerically investigate the pressure distribution and the diffusive capacity of an infinite long reservoir, whose cross section is the rectangle $[-10,20]\times[-10,10]$. The well is modeled as an infinite long cylindrical surface with square cross section of side length $0.5$, centered on the y-axis and rotated by 60 degrees around it. The fracture barycentric surface is given by $$\mathbf{r}(u,v)=\langle x_c,y_c,z_c \rangle + \left \langle  u, v, 2\sin(u)\right\rangle  \mbox{ with } (u,v)\in \left[0,2\pi\right]\times(-\infty,\infty),$$ and variable thickness $2h(u)=H(2+0.5\sin(7u))$. The vector $\langle x_c,y_c,z_c \rangle = \langle \frac{1}{8},0,\frac{\sqrt3}{8}\rangle$ 
has been chosen so that the fracture starts from the center of the top-right face of the well. 
Since the solution of the problem does not depend on $y$, we solve our equations only on the cross section of the domain given in Figures \ref{C1} (left and right).

\begin{figure}[h!]
\begin{minipage}{0.48\textwidth}
\includegraphics[width=1\linewidth,height=5.5cm]{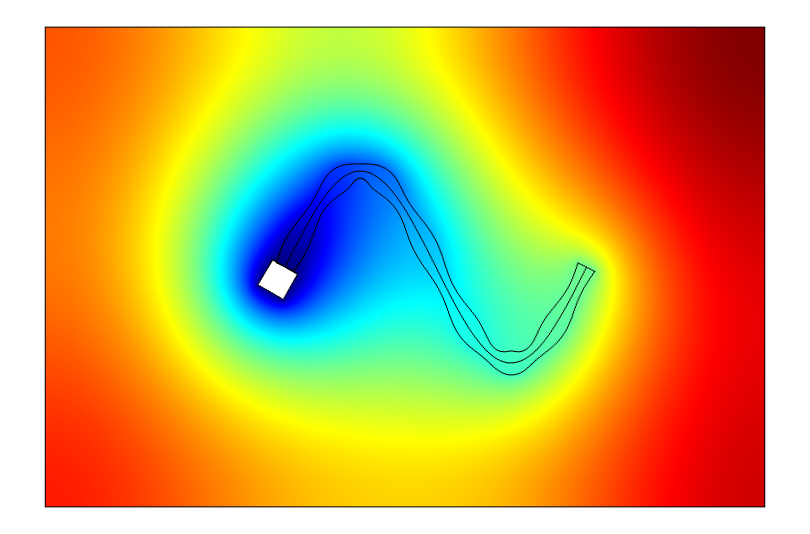}
\end{minipage}
\hspace{1mm}
\begin{minipage}{0.48\textwidth}
\includegraphics[width=1\linewidth,height=5.5cm]{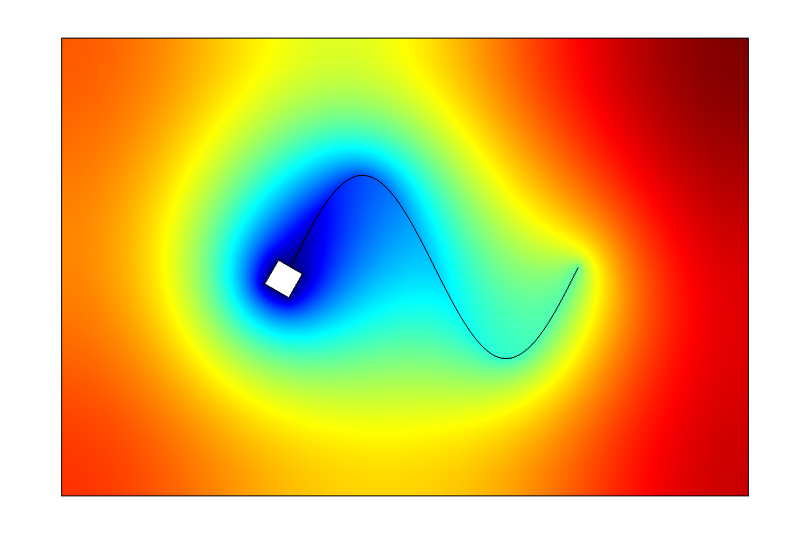}
\end{minipage}
\caption{{Coupled domain 1: pressure distributions in the coupled fracture porous media domain with $H=0.1$ and $\beta=0.001$, obtained using the original model (left) and the Reduced Model I (right), respectively.}}\label{C1}
\end{figure}

First, we couple the original flow equation inside the fracture  with the flow in the porous media, by imposing the continuity of the solutions and the continuity of the fluxes across the fracture boundaries. Zero Dirichlet boundary conditions are imposed on the well. 
Zero flux boundary conditions are imposed on all the outer boundaries of the reservoir and on the right end of the fracture. 
Then, we solve the coupled system using the flow equations in the Reduced Model I with zero Dirichlet boundary condition on the well and zero Neumann boundary condition on the right end of the fracture. 
% In there, the flow equation inside the fracture \eqref{FINAL} (on the barycentric line of the fracture cross section) is introduced 
% as a boundary condition in the porous media that takes into account the nonlinearity of the flow, 
% the changing thickness of the fracture and other geometric properties of the fracture. 

Figure \ref{C1}, presents the pressure distributions in the coupled domain obtained using the original model (left) and the Reduced Model I (right), for $H=0.1$ and $\beta=0.001$. The colors indicate that the fluid first converges towards the fracture and 
then flows towards the well. In Figure \ref{Fc1}, we compare the pressure distributions obtained using the original model 
(on the barycentric line of the fracture cross section) and the Reduced Model I, for $\beta=0$ (Darcy) and $\beta=100$ (Forchheimer). 
It is evident that the solutions of the two models are very close to each other. 

\begin{figure}[h!]
\begin{minipage}{0.49\textwidth}
\includegraphics[width=1\linewidth,height=5.5cm]{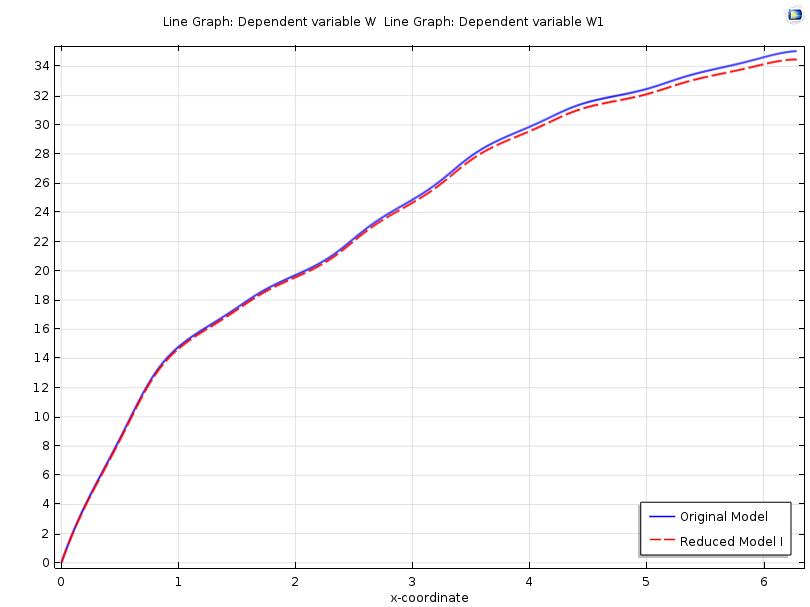}
\end{minipage}
\hspace{\fill}
\begin{minipage}{0.49\textwidth}
\includegraphics[width=1\linewidth,height=5.5cm]{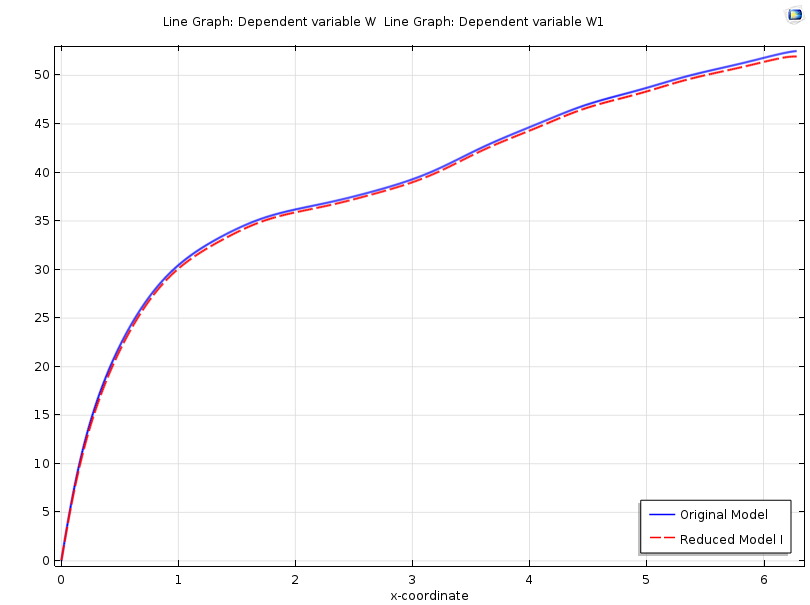}
\end{minipage}
\caption{Coupled domain 1: comparison between the pressure distributions in the coupled fracture porous media domain, obtained using the original model and the Reduced Model I, for $H=0.01$, $\beta=0$ (left) and $\beta=100$ (right), respectively.} \label{Fc1}
\end{figure}

Table \ref{Tc} presents the diffusive capacities in the coupled domain, obtained from the original model and the Reduced Model I, as $H$ and $\beta$ change. Moreover, the relative error, $\left|\frac{PI_{PD}-PI_{R_1}}{PI_{PD}}\right|$, has been reported. 
%The diffusive capacities obtained from both models increase with increasing fracture thickness, for all $\beta$. 
%At the same time, their values decrease, when the Forchheimer coefficient $\beta$ increases. This numerical results have a clear physical interpretation.
%Moreover, the error is small for all values of $H$ and $\beta$, and therefore the Reduced Model I can be effectively 
%used to approximate the original problem. It can be also observed that the error consistently increases with $H$, 
%while it has a slight dependence on $\beta$. 

\begin{table}[h] 
	\begin{center}
	\caption{Coupled domain 1: comparison of the diffusive capacities obtained from the original model and the Reduced Model I.} \label{Tc}
	\scalebox{0.64}{	
	\begin{tabular}{|c|c|c|c|c|c|c|c|c|c|}
	\hline	
	\backslashbox[2cm]{$\beta$}{\vspace{-4mm}H} &\multicolumn{3}{c|}{0.01}& \multicolumn{3}{c|}{0.05}& \multicolumn{3}{c|}{0.1}\\
	\hline
\hfill &	$PI_{PD}$ & $PI_{R_1}$ & error & $PI_{PD}$ & $PI_{R_1}$	& error	& $PI_{PD}$	& $PI_{R_1}$ & error\\
\hline
0	&0.035027478	&0.03524709	&6.27E-03	&0.05875387&	0.057609773	&1.95E-02	&0.076646162	&0.071405654	&6.84E-02\\
0.001	&0.034974637	&0.03519495	&6.30E-03	&0.058705992&	0.057564145	&1.95E-02	&0.076612623	&0.071375039	&6.84E-02\\
1	&0.027792767	&0.028039564	&8.88E-03	&0.041956574&	0.041597979	&8.55E-03	&0.058528659	&0.055027505	&5.98E-02\\
10	&0.025385268	&0.025591167	&8.11E-03	&0.030338957&	0.03057647	&7.83E-03	&0.037144807	&0.036224497	&2.48E-02\\
50	&0.024668646	&0.024847989	&7.27E-03	&0.026591578&	0.027038853	&1.68E-02	&0.02924881	&0.029500624	&8.61E-03\\
100	&0.024491447	&0.024661886	&6.96E-03	&0.025662443&	0.026149571	&1.90E-02	&0.027264275	&0.027825031&	2.06E-02
\\
\hline
\end{tabular} }
	\end{center}
\end{table}
}
\end{example}

\begin{example} \label{Ex7} {\rm
Next, we evaluate the diffusive capacity in an infinite long reservoir whose cross section is the rectangle $[-15,15]\times[-10,10]$. The well is modeled as an infinite long cylindrical surface centered on the $y$-axis, with rectangular cross section of height $0.5$ and width $20$. The reservoir contains three fractures whose geometries are identical, and are described by the barycentric surface given by $$\mathbf{r}_i(u,v)=\mathbf r_{0,i} + A_i \left\langle u,v,2\sin(u)\right\rangle \mbox{ with } (u,v)\in \left[0,2\pi\right]\times(-\infty,\infty), \; i=1,2,3, $$ and variable thickness $2h(u)=H(2+0.5\sin(7u))$.
In here $$\mathbf r_{0,1} = \langle 0, 0 , 0.25 \rangle,\quad \mathbf r_{0,2} = \langle -5, 0 ,-0.25 \rangle ,\quad \mathbf r_{0,3} = \langle 5, 0 ,-0.25 \rangle,$$
and $A_i$ are rotation matrices given by
$$A_i = \begin{bmatrix} 
\cos(\theta_i)& 0 & -\sin(\theta_i) \\
0 & 1 & 0 \\ 
\sin(\theta_i)& 0 & \cos(\theta_i) \end{bmatrix},  \; \mbox{ with } \; \theta_1 = \frac{\pi}{2}, \; \mbox{ and } \; \theta_2 = \theta_3 =\frac{3 \pi}{2}.  
$$

Again, since the solution of the problem does not depend on $y$, we solve our equations only on the cross section of the domain given in Figures \ref{H1} (left and right).
\begin{figure}[h!]
\begin{minipage}{0.48\textwidth}
\includegraphics[width=1\linewidth,height=5.5cm]{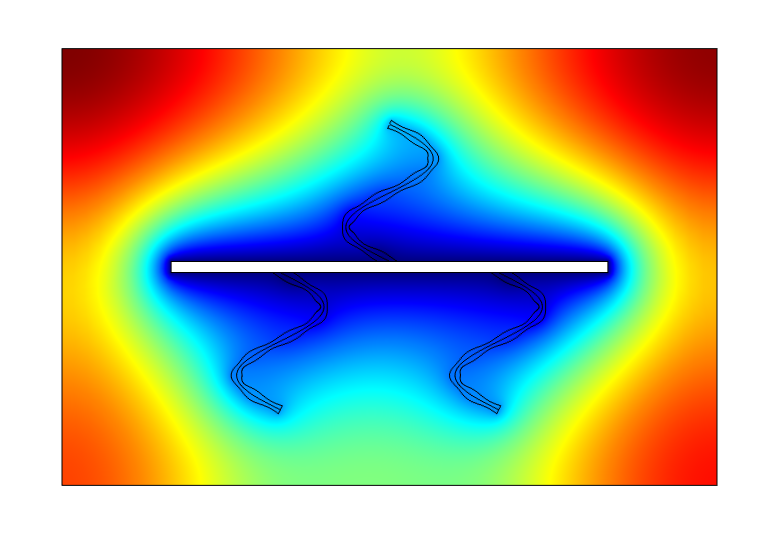}
\end{minipage}
\hspace{1mm}
\begin{minipage}{0.48\textwidth}
\includegraphics[width=1\linewidth,height=5.5cm]{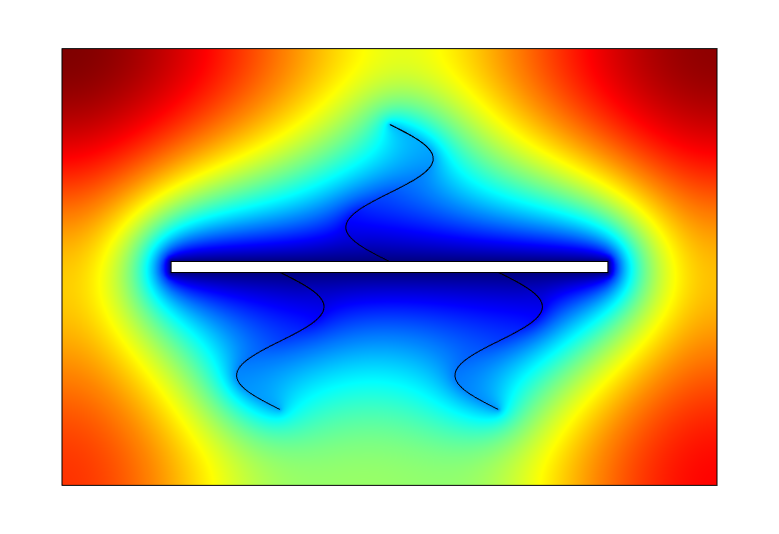}
\end{minipage}
\caption{Coupled domain 2: pressure distributions in the coupled fracture porous media domain with $H=0.1$ and $\beta=1$, obtained using the original model (left) and the Reduced Model I (right), respectively.}\label{H1}
\end{figure}

First, we couple the original flow equation inside the fracture with the flow in the porous media, by imposing the continuity of the solutions and the continuity of the fluxes across the fracture boundaries. Zero Dirichlet boundary conditions are imposed on the well. 
Zero flux boundary conditions are imposed on all the outer boundaries of the reservoir and on the ends of the fractures. 
Then, we solve the coupled system using the flow equations in the Reduced Model I, with zero Dirichlet boundary condition on the well and zero Neumann boundary condition on the ends of the fractures. 
%In there, the flow equation inside the fracture \eqref{FINAL} (on the barycentric line of the fracture cross section) is introduced 
%as a boundary condition in the porous media that takes into account the nonlinearity of the flow, 
%the changing thickness of the fracture and other geometric properties of the fracture. 

Figure \ref{H1} presents the pressure distributions in the coupled fracture porous media domain, obtained from original model (left) and the Reduced Mode I (right), respectively, for $H=0.1$ and $\beta=1$. The colors indicate that the fluid first converges towards the fracture and then flows towards the well.
Table \ref{Tc2} presents the diffusive capacities in the coupled domain, obtained using the original model and the Reduced Model I, as $H$ and $\beta$ change. Moreover, the relative error has been reported.
%We observe similar results as in the previous example. Namely, the diffusive capacities obtained from both models increase with increasing $H$, but decrease with increasing $\beta$. Moreover, the error is very small for all the values of $H$ and $\beta$, confirming that the Reduced Model I provides solutions that are similar to the solutions of the original model. The error consistently increases with increasing $H$, while the dependence of $\beta$ on the error is negligible.

\begin{table}[h]
	\begin{center}
	\caption{Coupled domain 2: comparison of the diffusive capacities obtained from the original model and the Reduced Model I.}\label{Tc2}
	\scalebox{0.64}{	
	\begin{tabular}{|c|c|c|c|c|c|c|c|c|c|}
	\hline	
	\backslashbox[2cm]{$\beta$}{\vspace{-4mm}H} &\multicolumn{3}{c|}{0.01}& \multicolumn{3}{c|}{0.05}& \multicolumn{3}{c|}{0.1}\\
	\hline
\hfill &	$PI_{PD}$ & $PI_{R_1}$ & error & $PI_{PD}$ & $PI_{R_1}$	& error	& $PI_{PD}$	& $PI_{R_1}$ & error\\
\hline
0	&0.149934761	&0.150403058&	3.12E-03&	0.184953216&	0.182806385	&1.16E-02	&0.209349275&	0.201457857	&3.77E-02\\
0.001	&0.149925428	&0.150393766&	3.12E-03&	0.184940273&	0.182794318&	1.16E-02&	0.209339858&	0.201449376&	3.77E-02\\
1	&0.145535944&	0.145975567&	3.02E-03&	0.176042698&	0.174430344	&9.16E-03&	0.201650306&	0.194507842&	3.54E-02\\
10	&0.140394828&	0.140694173&	2.13E-03&	0.157643454&	0.156802125&	5.34E-03&	0.177091286&	0.172336013&	2.69E-02\\
50&	0.138128196&	0.138326124&	1.43E-03&	0.147357331&	0.146777702	&3.93E-03&	0.158752511&	0.155725614&	1.91E-02\\
100	&0.137503318&	0.137668886&	1.20E-03&	0.144376243&	0.143869846	&3.51E-03&	0.153026912&	0.150547714&	1.62E-02\\
\hline
\end{tabular} }
	\end{center}
\end{table}

%	\begin{center}
%		\caption{Error of the Diffusive capacity, obtained using the Reduced Model II for a rectangular reservoir with a system of fractures.}
%	\begin{tabular}{|c|c|c|c|}
%	\hline
%\backslashbox[2cm]{\hspace{0.5mm}$\beta$}{\vspace{-4mm}H} &0.01& 0.05& 0.1\\
%\hline
%0 &  -7.07E-06	& 1.39E-02	& 4.23E-02\\
%0.001	& -7.36E-06 &	1.39E-02 & 4.23E-02 \\
%\hline
%\end{tabular} 
%	\end{center}
%\end{table}
\begin{table}[h]
	\end{table}}
\end{example}
Both Examples \ref{Ex6} and \ref{Ex7} show similar observations. Tables \ref{Tc} and \ref{Tc2} confirm that the diffusive capacities obtained from both models increase with increasing fracture thickness, for all $\beta$. 
At the same time, their values decrease, when the Forchheimer coefficient $\beta$ increases. This numerical results have a clear physical interpretation.
Moreover, we can see the error is small for all values of $H$ and $\beta$. The error increases with $H$, while it has a slight dependence on $\beta$. 
Obtained results show that, as $H\approx 10^{-2}$, the errors are very small, and therefore the Reduced Model I can be effectively used in large scale simulators for long and thin fractures with complicated geometries.

\section{Conclusion}

In this paper, we investigated the flow filtration process of slightly compressible fluids in porous media containing fractures with complex geometries. We modeled the coupled fractured porous media system where the linear Darcy flow is considered in porous media and the nonlinear Forchheimer equation is used inside the fracture. Using methods in differential geometry we formulated the fracture as a manifold immersed in Porous media. The equation for pressure of the flow inside the fracture was modeled and then a reduced model, where the fracture is presented as a boundary inside porous media, was obtained. Theoretical and numerical results were obtained to prove the closeness of the solutions of the actual model and the solutions of the reduced model, both inside the fracture and in the coupled domain.

The main conclusion can be formulated as: for actual field data, where the thickness of the fracture is small compared to the length of the fracture, the reduced model can be effectively
used in large scale simulators for long and thin fractures with complicated geometry.

%\pu{In future we intend to analyze the problem of optimizing the diffusive capacity, for fractures with complicated geometries.}
%\RED{Creating the geometry of the fractures may be challenging with current technology and resources in fracture reservoir simulations. However, the geometric method and the analysis we introduce in this paper can be utilized in ............... (micro fluid, arteries and veins in blood simulations bla bla ) Moreover, we believe that our methodology can be served as a foundation for future research and implementations in fracture reservoir modeling.}

Controlling the shape of the fractures in geological reservoirs is a challenging problem. Therefore our method can be applied mostly for simple fracture geometries. However, the geometric method and the analysis we introduced in this paper are valuable tools in modeling micro fluidic flows and blood flows in arteries and veins. Moreover, we believe that this methodology can be served as a foundation for reservoir engineers to model fractures in the future.
% with new innovative technology.

\section*{Acknowledgement}
This work was supported by the National Science Foundation grant NSF-DMS 1412796.

\bibliographystyle{plain}
\bibliography{pushpi}

\end{document}